\documentclass[11pt,reqno]{amsart}
\pdfoutput=1

\usepackage{latexsym,amsthm,amssymb,amsmath,amscd,hyperref}
\usepackage[shortlabels]{enumitem}
\usepackage[all,cmtip]{xy}
\usepackage{xcolor}

\usepackage{bm}

\usepackage{fancyhdr}

\newtheorem{theorem}{Theorem}[section]
\newtheorem{proposition}[theorem]{Proposition}
\newtheorem{lemma}[theorem]{Lemma}

\theoremstyle{remark}
\newtheorem{remark}[theorem]{Remark}

\newcommand{\Cl}{\textup{Cl}}

\newcommand{\spn}{\textup{span}}

 \renewcommand{\AA}{\mathbb{A}}
 \newcommand{\RR}{\mathbb{R}}
\newcommand{\E}{\mathcal{E}}
\newcommand{\I}{\mathcal{I}}
\newcommand{\J}{\mathcal{J}}
\renewcommand{\P}{\mathcal{P}}
\newcommand{\R}{\mathcal{R}}

\renewcommand{\a}{\mathfrak{a}}
\renewcommand{\b}{\mathfrak{b}}
\renewcommand{\c}{\mathfrak{c}}

\renewcommand{\k}{\mathfrak{k}}
\newcommand{\m}{\mathfrak{m}}
\newcommand{\p}{\mathfrak{p}}
\newcommand{\q}{\mathfrak{q}}

\begin{document}

\title{Phase transitions on C*-algebras from actions of congruence monoids on rings of algebraic integers}
\author{Chris Bruce}
\address[Chris Bruce]{
	Department of Mathematics and Statistics\\
	University of Victoria\\
	Victoria, BC V8W 2Y2\\
	Canada}%
\email[Bruce]{cmbruce@uvic.ca}

\subjclass[2010]{Primary 46L05; Secondary 11R04.} 
\thanks{Research supported by the Natural Sciences and Engineering Research Council of Canada through an Alexander Graham Bell CGS-D award.\\ This work was done as part of the author's PhD project at the University of Victoria.}

\maketitle

\setlength{\parindent}{0cm} \setlength{\parskip}{0.3cm}

\begin{abstract}
We compute the KMS (equilibrium) states for the canonical time evolution on C*-algebras from actions of congruence monoids on rings of algebraic integers. We show that for each $\beta\in[1,2]$, there is a unique KMS$_\beta$ state, and we prove that it is a factor state of type III$_1$. There are phase transitions at $\beta=2$ and $\beta=\infty$ involving a quotient of a ray class group. Our computation of KMS and ground states generalizes the results of Cuntz, Deninger, and Laca for the full $ax+b$-semigroup over a ring of integers, and our type classification generalizes a result of Laca and Neshveyev in the case of the rational numbers and a result of Neshveyev in the case of arbitrary number fields. 
\end{abstract}

\section{Introduction}

Given a number field $K$ with ring of integers $R$, Cuntz, Deninger, and Laca studied phase transitions for the canonical time evolution on the left regular C*-algebra $C_\lambda^*(R\rtimes R^\times)$ of the (full) $ax+b$-semigroup $R\rtimes R^\times$ over $R$, see \cite{CDL:2013}. 
Their results built on work of Laca and Raeburn in \cite{LR:2010} and Laca and Neshveyev in \cite{LN:2011} on the similar semigroup $\mathbb{N}\rtimes\mathbb{N}^\times$ associated to the number field $\mathbb{Q}$. The classification of KMS and ground states from \cite{CDL:2013} showed that the C*-dynamical system associated with $C_\lambda^*(R\rtimes R^\times)$ exhibits several interesting properties; for example, the parameterization spaces for both the ground states and the low temperature KMS states decomposes over the ideal class group $\Cl(K)$ of $K$, and uniqueness for the high temperature KMS states is related to the distribution of ideals over the group $\Cl(K)$. In \cite{Nesh:2013}, Neshveyev developed general results for computing KMS states on C*-algebras of non-principal groupoids, and gave an alternative computation of the KMS states on $C_\lambda^*(R\rtimes R^\times)$. 

The construction from \cite{CDL:2013} was recently generalized in \cite{Bru:2019} by restricting the multiplicative part of $R\rtimes R^\times$ to lie in certain subsemigroups of $R^\times$. Specifically, given a modulus $\m=\m_\infty\m_0$ for $K$ and a group $\Gamma$ of residues modulo $\m$, one considers the left regular C*-algebra $C_\lambda^*(R\rtimes R_{\m,\Gamma})$ of the semi-direct product $R\rtimes R_{\m,\Gamma}$ where $R_{\m,\Gamma}\subseteq R^\times$ is the congruence monoid consisting of algebraic integers in $R^\times$ that reduce to an element of $\Gamma$ modulo $\m$. For each number field $K$, the construction produces infinitely many non-isomorphic C*-algebras as $\m$ and $\Gamma$ vary.
For the special case of trivial $\m$, in which case $\Gamma$ must also be trivial, one gets the full $ax+b$-semigroup C*-algebra studied in \cite{CDL:2013}. And for the special case where $\m_\infty$ is supported at all real embeddings of $K$ and $\m_0$ is trivial, one gets the semigroup $R\rtimes R_+^\times$ where $R_+^\times$ is the subsemigroup of $R^\times$ consisting of (non-zero) totally positive algebraic integers.

The main result of this paper is the computation of all KMS and ground states on the left regular C*-algebra $C_\lambda^*(R\rtimes R_{\m,\Gamma})$ for the canonical time evolution $\sigma$ coming from the norm map on $K$, including a classification of type for the high temperature KMS states, see Theorem~\ref{thm:phasetransitions} for the precise statement. 
As a consequence, we obtain that the boundary quotient of $C_\lambda^*(R\rtimes R_{\m,\Gamma})$ admits a unique KMS state, which is of type III$_1$. Moreover, the techniques needed to prove Theorem~\ref{thm:phasetransitions} also lead to a computation of all the KMS and ground states on the left regular C*-algebras of the monoids $R_{\m,\Gamma}$ and $R_{\m,\Gamma}/R_{\m,\Gamma}^*$ where $R_{\m,\Gamma}^*:=R_{\m,\Gamma}^*\cap R^*$ is the group of units in $R_{\m,\Gamma}$.

In order to explain our main result, we must first discuss a few number-theoretic preliminaries, see Section~\ref{subsec:ANT} for more details.
Let $\I_\m$ denote the group of fractional ideals in $K$ that are coprime to the modulus $\m$, and let $K_{\m,\Gamma}=R_{\m,\Gamma}^{-1}R_{\m,\Gamma}\subseteq K^\times$ be the group of (left) quotients of $R_{\m,\Gamma}$. 
For each $x\in K_{\m,\Gamma}$, let $i(x):=xR$ be the principal fractional ideal in $K$ generated by $x$, so that $i(K_{\m,\Gamma})$ is a subgroup of $\I_\m$. The quotient group $\I_\m/i(K_{\m,\Gamma})$ will appear throughout this paper. In the case that $\m$ and $\Gamma$ are trivial, $\I_\m/i(K_{\m,\Gamma})$ coincides with the ideal class group $\Cl(K)$. In general, $\I_\m/i(K_{\m,\Gamma})$ is a quotient of the ray class group modulo $\m$ and thus is always a finite group.

We show that the C*-dynamical system $(C_\lambda^*(R\rtimes R_{\m,\Gamma}),\mathbb{R},\sigma)$ exhibits phase transitions at $\beta=2$ and $\beta=\infty$. For each $\beta\in[1,2]$, we prove that there is a unique $\sigma$-KMS$_\beta$ state on $C_\lambda^*(R\rtimes R_{\m,\Gamma})$. Our proof of uniqueness for $\beta\in[1,2]$ uses the well-known fact that the $L$-functions associated with non-trivial characters of $\I_\m/i(K_{\m,\Gamma})$ do not have poles at $1$. This technique is inspired by the proofs of uniqueness for high temperature KMS states on Bost--Connes type systems, see, for example, \cite[Section~7]{BC:1995}, \cite[Proposition]{Nesh:2002}, and \cite[Theorem 2.1(ii)]{LLN:2009}. 
To maneuver ourselves into a position where we can use these methods, we expand on an idea from \cite{Nesh:2013}. In the special case of trivial $\m$ and $\Gamma$, when our uniqueness result coincides with that in \cite[Theorem~6.7]{CDL:2013}, our approach is close to that taken in \cite[Section~3]{Nesh:2013} and is rather different than that taken in \cite{CDL:2013}.
We then prove that for each $\beta\in[1,2]$, the (unique) KMS$_\beta$ state $\phi_\beta$ on $C_\lambda^*(R\rtimes R_{\m,\Gamma})$ is a factor state of type III$_1$. Indeed, we prove that the von Neumann algebra generated by the GNS representation of $\phi_\beta$ is isomorphic to the injective factor of type III$_1$ with separable predual. This builds on \cite{LN:2011} and also generalizes the result asserted in \cite[Section~3]{Nesh:2013} on type for the high temperature KMS states on the left regular C*-algebra $C_\lambda^*(R\rtimes R^\times)$ of the full $ax+b$-semigroup, see Remark~\ref{rmk:typeIII1} below for more on this.
Our computation of the type uses ideas from \cite{LN:2011} and \cite{LagNesh:2014}; it relies on a general version of the prime ideal theorem for classes in $\I_\m/i(K_{\m,\Gamma})$. 

We now discuss the case where $\beta\in(2,\infty]$. For each class $\k\in\I_\m/i(K_{\m,\Gamma})$, choose an integral ideal $\a_\k\in\k$ representing $\k$. The group $R_{\m,\Gamma}^*$ of units of $R_{\m,\Gamma}$ acts on $\a_\k$ by multiplication, so we may form the semi-direct product $\a_\k\rtimes R_{\m,\Gamma}^*$. For each $\beta\in (2,\infty]$, we prove that the set of KMS$_\beta$ states decomposes over the finite set $\I_\m/i(K_{\m,\Gamma})$; specifically, the extremal KMS$_\beta$ states are parameterized by pairs $(\k,\tau)$ where $\k$ is a class in $\I_\m/i(K_{\m,\Gamma})$ and $\tau$ is an extremal tracial state on the group C*-algebra $C^*(\a_\k\rtimes R_{\m,\Gamma}^*)$. 
Moreover, the parameter space for the ground states also decomposes over $\I_\m/i(K_{\m,\Gamma})$, but the extreme points are given by pairs $(\k,\phi)$ where $\k\in \I_\m/i(K_{\m,\Gamma})$ and $\phi$ is an extremal state on a matrix algebra over $C^*(\a_\k\rtimes R_{\m,\Gamma}^*)$, so there are usually ground states that are not KMS$_\infty$ states. For the special case of trivial $\m$ and $\Gamma$, we recover the main parameterization results from \cite[Sections~7\&8]{CDL:2013}. Our computation uses general results for KMS and ground states on groupoid C*-algebras from \cite{Nesh:2013} and \cite{LLN:2019}, respectively.

The boundary quotient of $C_\lambda^*(R\rtimes R_{\m,\Gamma})$ also carries a canonical time evolution. We prove that the associated C*-dynamical system admits a unique KMS state, which is of type III$_1$ and has inverse temperature $\beta=1$.
In the special case where $\m$ and $\Gamma$ are trivial, the boundary quotient coincides with the ring C*-algebra of $R$ and we recover the known uniqueness result in that case, see  \cite{Cun:2008} for the case $K=\mathbb{Q}$ and \cite[Theorem~6.7]{CDL:2013} for the case of a general number field. 

The techniques and results used to prove Theorem~\ref{thm:phasetransitions} also lead to a phase transition theorem for the canonical time evolution on the left regular C*-algebra $C_\lambda^*(R_{\m,\Gamma})$ of a congruence monoid and also the left regular C*-algebra $C_\lambda^*(R_{\m,\Gamma}/R_{\m,\Gamma}^*)$ of the semigroup $R_{\m,\Gamma}/R_{\m,\Gamma}^*$ of principal integral ideals of $R$ that are generated by an element of $R_{\m,\Gamma}$. These simpler C*-dynamical systems also exhibit several phase transitions, and the group $\I_\m/i(K_{\m,\Gamma})$ also appears in this context. 
Additionally, there is a phase transition at $\beta=0$; the reason for this is that the spectrum of the diagonal in the case of the multiplicative monoids contains a unique fixed point, whereas in the case of $R\rtimes R_{\m,\Gamma}$, there are no fixed points.
This generalizes and expounds the result from \cite[Remark~7.5]{CDL:2013} (see also \cite[Remark~6.6.5]{CELY:2017}).

This paper is organized as follows. 
Section~\ref{sec:prelims} contains preliminaries: in Section~\ref{subsec:ANT}, we recall some well-known concepts from algebraic number theory, including that of moduli and ray class groups, and we fix some notation that will be used throughout this article; in Section~\ref{subsec:construction}, we review the necessary background on congruence monoids and C*-algebras from actions of congruence monoids on rings of algebraic integers from \cite{Bru:2019}.
In Section~\ref{sec:equilibrium}, we first introduce a canonical time evolution $\sigma$ on $C_\lambda^*(R\rtimes R_{\m,\Gamma})$, and state our main theorem on phase transitions; this result gives a parameterization of all KMS and ground states of the C*-dynamical system $(C_\lambda^*(R\rtimes R_{\m,\Gamma}),\mathbb{R},\sigma)$, including the type for all high temperature KMS states, see Theorem~\ref{thm:phasetransitions}. Sections \ref{subsec:prep} through \ref{subsec:ground} contain the proof of the parameterization results in Theorem~\ref{thm:phasetransitions}. The claim about type is proven in Section~\ref{sec:typeIII1}, see Theorem~\ref{thm:typeIII1}.
In Section~\ref{sec:boundaryquotient}, we use Theorem~\ref{thm:phasetransitions} to compute the KMS and ground states on the boundary quotient of $C_\lambda^*(R\rtimes R_{\m,\Gamma})$, see Theorem~\ref{thm:bqKMS}. 
Section~\ref{sec:multiplicativemonoid} contains our phase transition theorems for the left regular C*-algebras $C_\lambda^*(R_{\m,\Gamma})$ and $C_\lambda^*(R_{\m,\Gamma}/R_{\m,\Gamma}^*)$.

\subsection*{Acknowledgments.} 
I am grateful to my PhD supervisor, Marcelo Laca, for helpful discussions and for providing feedback on the content and style of this article. Some of the research for Section~\ref{sec:typeIII1} was done during the 2016 trimester program ``Von Neumann algebras'' at the Hausdorff Institute for Mathematics (HIM) in Bonn, and I would like to thank HIM and the organizers of this event for their hospitality and support. I would like to thank Brent Nelson for several helpful discussions on type III von Neumann algebras.

I am also indebted to one of the anonymous referees for their careful reading of the initial version of this article, for pointing out a gap in the original proof of Theorem~\ref{thm:phasetransitions}(iii) and an error in the original statement of Theorem~\ref{thm:ptmultiplicative}(iii), and for numerous helpful comments/suggestions which led to many improvements. In particular, the suggestion that Lemma~\ref{lem:general} be used in Section~\ref{sec:typeIII1} led to a streamlined proof of Proposition~\ref{prop:kmgergodic} and other suggestions led to a nicer proof of Theorem~\ref{thm:bqKMS}.

\section{Preliminaries}\label{sec:prelims}

\subsection{Moduli for number fields and ray class groups}\label{subsec:ANT}
Let $K$ be a number field with ring of integers $R$. Let $\P_K$ denote the set of non-zero prime ideals of $R$, so that each fractional ideal $\a$ of $K$ can be written as $\a=\prod_{\p\in\P_K}\p^{v_\p(\a)}$ where $v_\p(\a)\in\mathbb{Z}$ and $v_\p(\a)=0$ for all but finitely many $\p$. For $x\in K^\times:=K\setminus\{0\}$, the set $xR$ is the principal fractional ideal of $K$ generated by $x$, and write $v_\p(x)$ instead of $v_\p(xR)$.

Let $V_{K,\mathbb{R}}$ be the (finite) set of real embeddings $K\hookrightarrow\mathbb{R}$. A \emph{modulus} for $K$ is a function $\m:V_{K,\mathbb{R}}\sqcup\P_K\to \mathbb{N}$ such that 
\begin{itemize}
\item $\m_\infty:=\m\vert_{V_{K,\mathbb{R}}}$ takes values in $\{0,1\}$;
\item $\m\vert_{\P_K}$ is finitely supported, that is, $\m(\p)=0$ for all but finitely many $\p\in\P_K$.
\end{itemize}
Let $\m_0$ be the ideal of $R$ defined by $\m_0:=\prod_{\p\in\P_K}\p^{\m(\p)}$. It is conventional to write $\m$ as a formal product $\m=\m_\infty\m_0$, and to write $w\mid\m_\infty$ when $w\in V_{K,\mathbb{R}}$ is such that $\m(w)=1$. For background on moduli for number fields, we refer the reader to \cite[Chapter~V,~Section~1]{Mil:2013}. 

The \emph{multiplicative group of residues modulo $\m$} is 
\[
(R/\m)^*:=\left(\prod_{w\mid\m_\infty}\{\pm 1\}\right)\times (R/\m_0)^*
\]
where $ (R/\m_0)^*$ denotes the multiplicative group of units of the ring $R/\m_0$. We let $R^\times:=R\setminus\{0\}$ be the multiplicative semigroup of non-zero algebraic integers in $K$, and we let 
\[
R_\m:=\{x\in R^\times : v_\p(x)=0 \text{ for all } \p \text{ with } \p\mid\m_0\}
\] 
be the multiplicative semigroup of (non-zero) algebraic integers that are coprime to $\m_0$.
For $a\in R_\m$, let $[a]_\m$ denote the \emph{residue of $a$ modulo $\m$}
\[
[a]_\m:=((\textup{sign}(w(a)))_{w\mid\m_\infty},a+\m_0)\in (R/\m)^*
\]
where $\textup{sign}(w(a)):=w(a)/|w(a)|$. 

The map $a\mapsto [a]_\m$ extends uniquely to a surjective group homomorphism from the group of quotients $R_\m^{-1}R_\m$ of $R_\m$ onto $(R/\m)^*$, see \cite[Lemma~2.1]{Bru:2019}. By \cite[Lemma~2.2]{Bru:2019},  $R_\m^{-1}R_\m$ coincides with the group $K_\m:=\{x\in K^\times : v_\p(x)=0 \text{ for all } \p \text{ with } \p\mid\m_0\}$. 

 The \emph{ray modulo $\m$} is the kernel of the map $a\mapsto [a]_\m$, $K_\m\to (R/\m)^*$; it is denoted by $K_{\m,1}$.
 Let $\I_\m$ denote the group of fractional ideals of $K$ that are coprime to $\m_0$, and for $x\in K^\times$, we let $i(x):=xR$ denote the fractional ideal of $K$ generated by $x$; the \emph{ray class group modulo $\m$} is $\Cl_\m(K):=\I_\m/i(K_{\m,1})$. If $\m$ is trivial, that is, if $\m_\infty\equiv 0$ and $\m_0=R$, then $\I_\m$ equals the group $\I$ of all fractional ideals of $K$, and the ray modulo $\m$ is simply the multiplicative group $K^\times$ of non-zero elements in $K$. In this case, the ray class group modulo $\m$ coincides with the \emph{ideal class group} $\Cl(K)=\I/i(K^\times)$ of $K$.
 The canonical homomorphism is $\Cl_\m(K)\to\Cl(K)$ is surjective, and the ray class group $\Cl_\m(K)$ is always finite, see \cite[Chapter~V,~Theorem~1.7]{Mil:2013} for more on the relationship between $\Cl_\m(K)$ and $\Cl(K)$. 
 We mention in passing that ray class groups play an important role in the ideal-theoretic formulation of class field theory.

\subsection{C*-algebras from actions of congruence monoids on rings of integers}\label{subsec:construction}

We now recall the construction from \cite{Bru:2019}. Let $K$ be a number field with ring of integers $R$, and let $\m$ be a modulus for $K$. For a subgroup $\Gamma\subseteq (R/\m)^*$, let
\[
R_{\m,\Gamma}:=\{a\in R_\m : [a]_\m\in\Gamma\}.
\]
Then $R_{\m,\Gamma}$ is a multiplicative subsemigroup of $R^\times$; such semigroups are called \emph{congruence monoids} in the literature on semigroups, see, for example, \cite{GHK:2004} or \cite{GHK:2006}. 
By \cite[Proposition~3.1]{Bru:2019}, the group of quotients $R_{\m,\Gamma}^{-1}R_{\m,\Gamma}$ coincides with 
\[
K_{\m,\Gamma}:=\{x\in K_\m : [x]_\m\in\Gamma\}.
\]
The group $i(K_{\m,\Gamma})$ of principal fractional ideals generated by elements of $K_{\m,\Gamma}$ has finite index in $\I_\m$; indeed, the quotient $\I_\m/i(K_{\m,\Gamma})$ can be canonically identified with the quotient $\Cl_\m(K)/\bar{\Gamma}$ where $\bar{\Gamma}:=i(K_{\m,\Gamma})/i(K_{\m,1})$.

The semigroup $R_{\m,\Gamma}$ acts on (the additive group of) $R$ by multiplication, so we may form the semi-direct product semigroup $R\rtimes R_{\m,\Gamma}$. For each $(b,a)\in R\rtimes R_{\m,\Gamma}$, let $\lambda_{(b,a)}$ be the isometry in $\mathcal{B}(\ell^2(R\rtimes R_{\m,\Gamma}))$ determined by $\lambda_{(b,a)}(\varepsilon_{(y,x)})=\varepsilon_{(b+ay,ax)}$ where $\{\varepsilon_{(y,x)}: (y,x)\in R\rtimes R_{\m,\Gamma}\}$ is the canonical orthonormal basis for $\ell^2(R\rtimes R_{\m,\Gamma})$.
The \emph{left regular C*-algebra $C_\lambda^*(R\rtimes R_{\m,\Gamma})$ of $R\rtimes R_{\m,\Gamma}$} is the sub-C*-algebra of $\mathcal{B}(\ell^2(R\rtimes R_{\m,\Gamma}))$ generated by the left regular representation of $R\rtimes R_{\m,\Gamma}$. That is,
\[
C_\lambda^*(R\rtimes R_{\m,\Gamma}):=C^*(\{\lambda_{(b,a)} : (b,a)\in R\rtimes R_{\m,\Gamma}\}).
\]
We refer the reader to \cite{Li1:2012,Li2:2013} or \cite[Chapter~5]{CELY:2017} for the general theory of semigroup C*-algebras. Let $\I_\m^+$ denote the non-zero integral ideals of $R$ that are coprime to $\m_0$, and for each $\a\in\I_\m^+$ and $x\in R$, let $E_{(x+\a)\times (\a\cap R_{\m,\Gamma})}$ be the orthogonal projection from $\ell^2(R\rtimes R_{\m,\Gamma})$ onto the subspace $\ell^2((x+\a)\times (\a\cap R_{\m,\Gamma}))$.
The C*-algebra $C_\lambda^*(R\rtimes R_{\m,\Gamma})$ has a canonical ``diagonal'' sub-C*-algebra given by  
\[
D_\lambda(R\rtimes R_{\m,\Gamma})=\overline{\spn}(\{E_{(x+\a)\times (\a\cap R_{\m,\Gamma})} : x\in R, \a\in \I_\m^+\}),
\]
see \cite[Proposition~3.3]{Bru:2019} and \cite[Section~3]{Li1:2012}.

\section{Equilibrium states}\label{sec:equilibrium}

\subsection{The phase transition theorem}
Let $B$ be a C*-algebra. A \emph{time evolution} on $B$ is a group homomorphism $\gamma:\mathbb{R} \to \textup{Aut}(B)$ such that for each fixed $x\in B$, the map $t\mapsto \gamma_t(x)$ is continuous. The triple $(B,\mathbb{R},\gamma)$ is called a \emph{C*-dynamical system}. There is a standard notion of equilibrium in this context, namely that of KMS and ground states. We now recall the relevant definitions.

Let $(B,\mathbb{R},\gamma)$  be a C*-dynamical system. An element $x \in B$ is \emph{$\gamma$-analytic} if the map $\mathbb{R}\to B$ given by $t \mapsto \gamma_t(x)$ extends to an entire function $z\mapsto \gamma_z(x)$ from $\mathbb{C}$ to $B$. Let $\beta\in \mathbb{R}^*:=\mathbb{R}\setminus\{0\}$. A state $\varphi$ on $B$ is a \emph{$\gamma$-KMS$_\beta$ state}, or a \emph{KMS state at inverse temperature $\beta$ for $\gamma$}, if it satisfies the \emph{KMS$_\beta$ condition}
\begin{equation}\label{KMScondition}
\varphi(xy) = \varphi(y \gamma_{i\beta}(x))
\end{equation}
for all $\gamma$-analytic elements $x,y$ in a $\gamma$-invariant subset with dense linear span. The parameter $\beta$ is often called the \emph{inverse temperature}, see \cite[Chapter~5]{BR:1997} for motivation from quantum statistical mechanics. 
The \emph{$\gamma$-KMS$_0$ states} are defined to be the $\gamma$-invariant traces on $B$; these are the ``infinite temperature'' or ``chaotic'' states. 

If $B$ is unital, then \cite[Theorem~5.3.30(1)\&(2)]{BR:1997} asserts that for each $\beta\in\mathbb{R}$, the set $\Sigma_\beta$ of KMS$_\beta$ states of the system $(B,\mathbb{R},\gamma)$ is a (possibly empty) convex weak$^*$-compact subset of the state space $S(B)$ of $B$ that is also a Choquet simplex. Moreover, by \cite[Theorem~5.3.30(3)]{BR:1997}, a KMS$_\beta$ state $\phi$ is an extreme point of $\Sigma_\beta$ if and only if $\phi$ is a factor state, that is, if the von Neumann algebra $\pi_\phi(B)''$ generated by the GNS representation $\pi_\phi$ of $\phi$ is a factor.

For $\beta=\infty$, that is, for ``zero temperature'', there are two different notions of equilibrium states. A state $\varphi$ on $B$ is a \emph{$\gamma$-KMS$_\infty$-state} if it is the weak*-limit of a net $(\varphi_i)_i$ where $\varphi_i$ is a $\gamma$-KMS$_{\beta_i}$-state for each $i$ and $\beta_i\to\infty$, and a state $\varphi$ on $B$ is a \emph{$\gamma$-ground state} if the map 
\begin{equation*}
z\mapsto \varphi(x\gamma_z(y))
\end{equation*}
is bounded on the upper half-plane for all $\gamma$-analytic elements $x,y$ in a $\gamma$-invariant subset with dense linear span. 
Every KMS$_\infty$ state is a ground state by \cite[Proposition~5.3.23]{BR:1997}, but there may be ground states that are not KMS$_\infty$ states, as we shall see in Theorem~\ref{thm:phasetransitions}(iii)\&(iv) below. Also see \cite[Corollary~1.8]{LLN:2019} for a more general explanation of why the set of  KMS$_\infty$ states may be properly contained in the set of ground states, and \cite[Theorem~7.1(3)\&(4)]{LR:2010} and \cite[Section~8]{CDL:2013} for examples of this phenomenon.
Note that the distinction between KMS$_\infty$ states and ground states was first made in \cite[Definition~3.7]{CM:2008}, and is not observed in \cite{BR:1997}.

If $B$ is unital, then the set $\Sigma_\infty$ of KMS$_\infty$ states of the system $(B,\mathbb{R},\gamma)$ is a convex weak$^*$-compact subset of $S(B)$ by \cite[Proposition~3.8]{CM:2008}, whereas the set of ground states need not be a simplex, see \cite[Remark~7.2(v)]{LR:2010}.

We now return to the C*-algebra $C_\lambda^*(R\rtimes R_{\m,\Gamma})$. For a non-zero ideal $\a$ of $R$, we let $N(\a):=|R/\a|$ denote the norm of $\a$, which is always finite. The map $a\mapsto N(aR)$ defines a semigroup homomorphism from $R^\times$ to the multiplicative semigroup $\mathbb{N}^\times:=\mathbb{N}\setminus\{0\}$ of positive integers, and the map $R\rtimes R_{\m,\Gamma}\to \mathbb{R}_+^*$ given by $(b,a)\mapsto N(a)$ is a semigroup homomorphism. For each $t\in\mathbb{R}$, let $U_t$ denote the diagonal unitary on $\ell^2(R\rtimes R_{\m,\Gamma})$ that is determined on the canonical basis by
\[
U_t(\varepsilon_{(b,a)})=N(a)^{it}\varepsilon_{(b,a)}.
\] 

Then $t\mapsto U_t$ defines a unitary representation $\mathbb{R}\to\mathcal{U}(\ell^2(R\rtimes R_{\m,\Gamma}))$, and a routine argument shows that the group $\{U_t : t\in\mathbb{R}\}$ of unitaries implements a time evolution on $C_\lambda^*(R\rtimes R_{\m,\Gamma})$; specifically, we have the following result.

\begin{proposition}\label{prop:timeevolution}
	There is a time evolution $\sigma:\mathbb{R}\to\textup{Aut}(C_\lambda^*(R\rtimes R_{\m,\Gamma}))$ such that 
	\begin{equation}\label{eqn:timeevolution}
	\sigma_t(\lambda_{(b,a)})=N(a)^{it}\lambda_{(b,a)} \quad  \text{ for all } (b,a)\in  R\rtimes R_{\m,\Gamma}\text{ and } t\in\mathbb{R}.
	\end{equation}
\end{proposition}

Let $\k\in\I_\m/i(K_{\m,\Gamma})$ be an ideal class. An integral ideal $\a\in\k$ is said to be \emph{norm-minimizing in the class $\k$} if $N(\a)\leq N(\b)$ for every other integral ideal $\b$ in $\k$; note that each ideal class $\k$ contains only finitely many norm-minimizing ideals.
Norm-minimizing ideals appeared in \cite{LvF:2006} during the investigation of phase transitions for C*-dynamical systems associated with Hecke C*-algebras, and then later in \cite[Section~8]{CDL:2013} and \cite{LLN:2019}.

Let $R_{\m,\Gamma}^*:=R_{\m,\Gamma}\cap R^*$ be the group of invertible elements in $R_{\m,\Gamma}$. For each fractional ideal $\a\in\I_\m$, the group $R_{\m,\Gamma}^*$ acts on (the additive group of) $\a$ by multiplication, so we may form the semi-direct product group $\a\rtimes R_{\m,\Gamma}^*$. 

View $\ell^\infty(R\rtimes R_{\m,\Gamma})$ as a sub-C*-algebra of $\mathcal{B}(\ell^2(R\rtimes R_{\m,\Gamma}))$ in the canonical way, and let $\E$ be the restriction to $C_\lambda^*(R\rtimes R_{\m,\Gamma})$ of the canonical faithful conditional expectation $\mathcal{B}(\ell^2(R\rtimes R_{\m,\Gamma}))\to \ell^\infty(R\rtimes R_{\m,\Gamma})$. It follows from \cite[Lemma~3.11]{Li1:2012} that the range of $\E$ is equal to $D_\lambda(R\rtimes R_{\m,\Gamma})$. 
The main result of this paper is the following phase transition theorem.

\begin{theorem}\label{thm:phasetransitions}
Let $K$ be a number field, $\m$ a modulus for $K$, and $\Gamma$ a subgroup of $(R/\m)^*$. For each $\k\in\I_\m/i(K_{\m,\Gamma})$, choose a norm-minimizing ideal $\a_{\k,1}$ in $\k$. 
\begin{enumerate}[\upshape(i)]
	\item There are no $\sigma$-KMS$_\beta$ states on $C_\lambda^*(R\rtimes R_{\m,\Gamma})$ for $\beta<1$.
	\item For each $\beta\in [1,2]$, there is a unique $\sigma$-KMS$_\beta$ state $\phi_\beta$ on $C_\lambda^*(R\rtimes R_{\m,\Gamma})$. The state $\phi_\beta$ factors through the expectation $\E:C_\lambda^*(R\rtimes R_{\m,\Gamma})\to D_\lambda(R\rtimes R_{\m,\Gamma})$ and is determined by the values 
		\[
		\phi_\beta(E_{(x+\a)\times(\a\cap R_{\m,\Gamma})})=N(\a)^{-\beta}\quad \text{for } x\in R \text{ and } \a\in\I_\m^+.
		\] 
		Moreover, $\phi_\beta$ is of type III$_1$; indeed, the von Neumann algebra $\pi_{\phi_\beta}(C_\lambda^*(R\rtimes R_{\m,\Gamma}))''$ generated by the GNS representation $\pi_{\phi_\beta}$ of $\phi_\beta$ is isomorphic to the injective factor of type III$_1$ with separable predual.
    \item For each $\beta\in (2,\infty)$, there is an affine isomorphism of the simplex of tracial states on the C*-algebra 
    \[
    \bigoplus_{\k\in \I_\m/i(K_{\m,\Gamma})} C^*(\a_{\k,1}\rtimes R_{\m,\Gamma}^*)
    \] 
    onto the simplex of $\sigma$-KMS$_\beta$ states on $C_\lambda^*(R\rtimes R_{\m,\Gamma})$. 
	\item There is an affine isomorphism of the $\sigma$-ground state space of $C_\lambda^*(R\rtimes R_{\m,\Gamma})$ onto the state space of the C*-algebra 
	\[
		\bigoplus_{\k\in \I_\m/i(K_{\m,\Gamma})} M_{k_\k\cdot N(\a_{\k,1})}(C^*(\a_{\k,1}\rtimes R_{\m,\Gamma}^*))
	\]
	 where $k_\k$ is the number of norm-minimizing ideals in the class $\k$.
\end{enumerate}
\end{theorem}

Before continuing to the proof, we make several remarks.

\begin{remark}
\begin{itemize}
\item[(a)] For the particular case of trivial $\m$ and $\Gamma$, Theorem~\ref{thm:phasetransitions} recovers the parameterization results obtained in \cite[Sections~6~and~7]{CDL:2013}. 
\item[(b)] If $\a$ and $\b$ lie in the same class $\k\in\I_\m/i(K_{\m,\Gamma})$, so that there is a $k\in K_{\m,\Gamma}$ with $\a=k\b$, then the map $x\mapsto kx$ defines an $R_{\m,\Gamma}^*$-equivariant isomorphism $\a\cong \b$, so that $\a\rtimes R_{\m,\Gamma}^*\cong \b\rtimes R_{\m,\Gamma}^*$. Thus, in parts (iii) and (iv), we could replace each C*-algebra $C^*(\a_{\k,1}\rtimes R_{\m,\Gamma}^*)$ with $C^*(\a_\k\rtimes R_{\m,\Gamma}^*)$ for any other ideal $\a_\k$ in the class $\k$. 
\item[(c)]  In light of Theorem~\ref{thm:phasetransitions}(iii), it is natural to ask if one can explicitly describe the simplex of traces on the group C*-algebra $C^*(\a_\k\rtimes R_{\m,\Gamma}^*)$ for a fixed class $\k\in \I_\m/i(K_{\m,\Gamma})$ and integral ideal $\a_\k\in\k$. It turns out that, even for trivial $\m$ and $\Gamma$, this is a difficult problem that is related to the generalized Furstenberg conjecture, see \cite{LW:2018} and \cite[\S~5]{BLT}. 
\item[(d)] In joint work with Xin Li \cite[Theorem~4.1]{BruLi:2019}, we prove that the $K$-theory of the C*-algebra $C_\lambda^*(R\rtimes R_{\m,\Gamma})$ decomposes as
\[
K_*(C_\lambda^*(R\rtimes R_{\m,\Gamma}))\cong\bigoplus_{\k\in \I_\m/i(K_{\m,\Gamma})} K_*(C^*(\a_\k\rtimes R_{\m,\Gamma}^*))
\]
where $\a_\k$ is an integral ideal in the class $\k$. 
It is interesting that the C*-algebra $\bigoplus_{\k\in \I_\m/i(K_{\m,\Gamma})} C^*(\a_\k\rtimes R_{\m,\Gamma}^*)$ appears in both the $K$-theory formula and the parameterization of the low temperature KMS states. For the $ax+b$-semigroup $R\rtimes R^\times$, this has already been discussed by Cuntz in \cite[Chapter~6,~Section~6]{CELY:2017}.
\item[(e)] An alternative method for computing the low temperature KMS states on $C_\lambda^\times(R\rtimes R^\times)$ is discussed in \cite[Chapter~6,~Section~6]{CELY:2017}. Presumably, it could also be used here.
\item[(f)] It follows from Theorem~\ref{thm:phasetransitions}(iii)\&(iv) that there are usually ground states which are not KMS$_\infty$ states, that is, there is a phase transition at $\beta=\infty$. 
For example, let $K=\mathbb{Q}$, so that $R=\mathbb{Z}$. Let $m\in\mathbb{N}^\times$ be a positive natural number, and let $\m=\m_\infty\m_0$ where $\m_\infty$ takes the value one at the only real embedding of $\mathbb{Q}$ and $\m_0(p):=v_p(m)$. 
Then a calculation shows that the map  $(\mathbb{Z}/m\mathbb{Z})^*\to \I_\m/i(K_{\m,1})$ given by $[a]_m\mapsto [a\mathbb{Z}]$
is a well-defined isomorphism $(\mathbb{Z}/m\mathbb{Z})^*\cong\I_\m/i(K_{\m,1})$ and that each class $\k\in \I_\m/i(K_{\m,1})$ contains a unique norm-minimizing ideal of norm $n_\k$ where $n_\k$ is the smallest positive integer in the residue class modulo $m$ corresponding to $\k$ under the above isomorphism.
Moreover, in this situation, the isotropy groups appearing in Theorem~\ref{thm:phasetransitions}(iii)\&(iv) are all isomorphic to $\mathbb{Z}$, so Theorem~\ref{thm:phasetransitions} implies that the KMS$_\infty$ states are parameterized by traces on the commutative C*-algebra
\[
\bigoplus_{\k\in(\mathbb{Z}/m\mathbb{Z})^*}C^*(\mathbb{Z})\cong \bigoplus_{\k\in(\mathbb{Z}/m\mathbb{Z})^*}C(\mathbb{T}),
\]
whereas the ground states are parameterized by states on the C*-algebra
\[
\bigoplus_{\k\in(\mathbb{Z}/m\mathbb{Z})^*}M_{n_\k}(C^*(\mathbb{Z}))\cong\bigoplus_{\k\in(\mathbb{Z}/m\mathbb{Z})^*}M_{n_\k}(C(\mathbb{T})).
\]
\item[(g)] For $\beta>2$, the extremal $\sigma$-KMS$_\beta$ states on $C_\lambda^*(R\rtimes R_{\m,\Gamma})$ are either type I or type II. However, the techniques needed to deal with the case $\beta>2$ are rather different since these states usually do not factor through the expectation $\E$, see \cite{BLT}
\end{itemize}
\end{remark}

This section and the next are devoted to the proof of Theorem~\ref{thm:phasetransitions}, which we break up into several parts. The next five subsections contain some preliminaries, and the proofs of parts (i) through (iv), excluding the type computation. The proof that the von Neumann algebra $\pi_{\phi_\beta}(C_\lambda^*(R\rtimes R_{\m,\Gamma}))''$ is isomorphic to the injective factor of type III$_1$ with separable predual is given in Section~\ref{sec:typeIII1}.

\subsection{Preliminaries for the proof}\label{subsec:prep}
The semigroup $R\rtimes R_{\m,\Gamma}$ canonically embeds into the group $(R_\m^{-1}R)\rtimes K_{\m,\Gamma}$ where $(R_\m^{-1}R)=\{\frac{a}{b} : a\in R, b\in R_\m\}$ is the localization of $R$ at $R_\m$. By \cite[Proposition~3.2]{Bru:2019}, the semigroup $R\rtimes R_{\m,\Gamma}$ is left Ore and its group of left quotients coincides with $(R_\m^{-1}R)\rtimes K_{\m,\Gamma}$. That is, $(R\rtimes R_{\m,\Gamma})^{-1}(R\rtimes R_{\m,\Gamma})=(R_\m^{-1}R)\rtimes K_{\m,\Gamma}$.
To simplify notation, let 
\[
P_{\m,\Gamma}:=R\rtimes R_{\m,\Gamma} \quad\text{and}\quad G_{\m,\Gamma}:=(R_\m^{-1}R)\rtimes K_{\m,\Gamma}.
\] 
Also let $S:=\{\p\in\P_K : \p\mid \m_0\}$ be the \emph{support of $\m_0$}, which is a finite set of primes, and put $\mathcal{P}_K^\m:=\P_K\setminus S$.

\subsubsection{A groupoid model} The material below on adeles and a groupoid model for $C_\lambda^*(P_{\m,\Gamma})$ is from \cite[Section~5]{Bru:2019}. It was motivated by similar results from \cite[Section~5]{CDL:2013} for the special case where $\m$ and $\Gamma$ are trivial. 

For each non-zero prime ideal $\p$ of $R$, let $K_\p$ be the corresponding $\p$-adic completion of $K$ and $R_\p$ the ring of integers in $K_\p$. Let
\[
\mathbb{A}_S:=\Big\{\mathbf{a}=(a_\p)_\p\in \prod_{\p\in \P_K^\m}K_\p : a_\p\in R_\p \text{ for all but finitely many }\p\Big\}
\]
equipped with the restricted product topology with respect to the compact open subsets $R_\p\subseteq K_\p$. 
Denote by $\hat{R}_S$ the compact subring $\prod_{\p\in \P_K^\m}R_\p$, and let $\hat{R}_S^*:=\prod_{\p\in \P_K^\m}R_\p^*$ be the group of units of $\hat{R}_S$. The compact group $\hat{R}_S^*$ acts on $\mathbb{A}_S$ by multiplication, and we let $\bar{\mathbf{a}}$ denote the image of $\mathbf{a}\in\mathbb{A}_S$ under the quotient mapping $\mathbb{A}_S\to \mathbb{A}_S/\hat{R}_S^*$.
Define an equivalence relation on $\mathbb{A}_S\times\mathbb{A}_S/\hat{R}_S^*$ by 
\[
(\mathbf{b},\bar{\mathbf{a}})\sim(\mathbf{d},\bar{\mathbf{c}}) \quad \text{if } \bar{\mathbf{a}}=\bar{\mathbf{c}} \text{ and } \mathbf{b}-\mathbf{d}\in \bar{\mathbf{a}}\hat{R}_S.
\] 
Via the diagonal embedding, the groups $R_\m^{-1}R_\m$ and $K_{\m,\Gamma}$ act on $\mathbb{A}_S$ by translation and multiplication, respectively. The canonical action of $G_{\m,\Gamma}$ on $\mathbb{A}_S\times\mathbb{A}_S/\hat{R}_S^*$ given by $(n,k)(\mathbf{b},\bar{\mathbf{a}})=(n+k\mathbf{b},k\bar{\mathbf{a}})$ descends to a well-defined action on the locally compact Hausdorff quotient space 
\[
\Omega_K^\m:=(\mathbb{A}_S\times\mathbb{A}_S/\hat{R}_S^*)/\sim.
\]
By restricting the above equivalence relation to the subset $\hat{R}_S\times\hat{R}_S/\hat{R}_S^*\subseteq\mathbb{A}_S\times\mathbb{A}_S/\hat{R}_S^*$, we obtain the compact open subset
\[
\Omega_R^\m:=(\hat{R}_S\times\hat{R}_S/\hat{R}_S^*)/\sim
\] 
of $\Omega_K^\m$. Let $G_{\m,\Gamma}\ltimes\Omega_R^\m$ be the reduction of the transformation groupoid $G_{\m,\Gamma}\ltimes\Omega_K^\m$ by the set $\Omega_R^\m$, that is,
\[
G_{\m,\Gamma}\ltimes\Omega_R^\m=\{(g,w)\in G_{\m,\Gamma}\ltimes\Omega_K^\m : gw\in \Omega_R^\m\}.
\] 
Our choice of notation for the reduction groupoid comes from the fact that $G_{\m,\Gamma}\ltimes\Omega_R^\m$ can be canonically identified with a \emph{partial transformation groupoid}, see \cite[Section~3.3]{Li:2017}.

\begin{proposition}[{\cite[Propositions~4.1~and~5.3]{Bru:2019}}]\label{prop:gpoidmodel}
There is an isomorphism
\[
\vartheta:C_\lambda^*(P_{\m,\Gamma})\cong C^*(G_{\m,\Gamma}\ltimes\Omega_R^\m)
\]
that is determined on generators by $\vartheta(\lambda_{(b,a)})=1_{\{(b,a)\}\times\Omega_R^\m}$ for $(b,a)\in P_{\m,\Gamma}$.
\end{proposition}
\begin{proof}
The key result needed here is \cite[Proposition~5.3]{Bru:2019}. The proof briefly sketched in \cite{Bru:2019} follows arguments from \cite[Section~2]{Li:2014} closely. Here, we shall sketch another more direct proof, which is closer to the proof of the analogous result given in \cite[Section~5]{CDL:2013} for the case of the full $ax+b$-semigroup.

Using Li's theory of semigroup C*-algebras, it is shown in \cite[Section~5]{Bru:2019} that $C_\lambda^*(P_{\m,\Gamma})$ can be canonically identified with the groupoid C*-algebra $C_r^*(G_{\m,\Gamma}\ltimes \Omega_{P_{\m,\Gamma}})$ of the partial transformation groupoid $G_{\m,\Gamma}\ltimes \Omega_{P_{\m,\Gamma}}$ where $\Omega_{P_{\m,\Gamma}}$ is the spectrum of the commutative C*-algebra 
\[
D_\lambda(P_{\m,\Gamma})=\overline{\spn}(\{E_{(x+\a)\times (\a\cap R_{\m,\Gamma})} : x\in R, \a\in\I_\m^+\}),
\]
$E_{(x+\a)\times (\a\cap R_{\m,\Gamma})}\in\mathcal{B}(\ell^2(R\rtimes R_{\m,\Gamma}))$ is the orthogonal projection onto the subspace $\ell^2((x+\a)\times (\a\cap R_{\m,\Gamma}))$, and the partial action of $G_{\m,\Gamma}$ on $D_\lambda(P_{\m,\Gamma})$ is determined by $(n,k)E_{(x+\a)\times (\a\cap R_{\m,\Gamma})}=E_{(n+kx+k\a)\times (k\a)\cap R_{\m,\Gamma}}$ for $(n,k)\in G_{\m,\Gamma}$ and $x\in R$, $\a\in\I_\m^+$ such that $k\a\in\I_\m^+$ and $n+kx\in R$.

For each coset $x+\a$ where $\a\in\I_\m^+$ and $x\in R$, let 
\[
V_{(x+\a)\times\a^\times}:=\{[\mathbf{b},\bar{\mathbf{a}}]\in \Omega_R^\m : v_\p(\bar{\mathbf{a}})\geq v_\p(\a), v_\p(\mathbf{b}-x)\geq v_\p(\a) \text{ for all }\p\in\P_K^\m\}
\]
where $\a^\times:=\a\setminus\{0\}$. Then a calculation shows that $V_{(x+\a)\times\a^\times}\cap V_{(y+\b)\times\b^\times}=V_{[(x+\a)\cap (y+\b)]\times (\a\cap\b)^\times}$ where $V_\emptyset:=\emptyset$. 
By \cite[Proposition~3.4]{Bru:2019} and \cite[Corollary~2.7]{Li2:2013}, there exists a *-homomorphism $D_\lambda(P_{\m,\Gamma})\to C(\Omega_R^\m)$ such that $E_{(x+\a)\times (\a\cap R_{\m,\Gamma})}\mapsto 1_{V_{(x+\a)\times\a^\times}}$. This map is injective by \cite[Proposition~5.6.21]{CELY:2017}, and one can directly check, as was done in \cite[Proposition~5.2]{CDL:2013} for the full $ax+b$-semigroup, that this map is also surjective. Since it is also $G_{\m,\Gamma}$-equivariant, it follows that $C_r^*(G_{\m,\Gamma}\ltimes \Omega_{P_{\m,\Gamma}})\cong C^*(G_{\m,\Gamma}\ltimes\Omega_R^\m)$, and it is not difficult to see that $\lambda_{(b,a)}$ is mapped to $1_{\{(b,a)\}\times\Omega_R^\m}$ for all $(b,a)\in P_{\m,\Gamma}$.
\end{proof}

\subsubsection{Quasi-invariant measures on $\Omega_R^\m$}

The multiplicative map $R^\times\to \mathbb{N}^\times$ given by $a\mapsto N(aR)=|R/aR|$ has a unique extension to a group homomorphism $K^*\to\mathbb{Q}_+^*$ that we also denote by $N$. Let $c_N$ be the real-valued one-cocycle $G_{\m,\Gamma}\ltimes \Omega_R^\m\to \mathbb{R}$ given by $c_N((n,k),w)=\log N(k)$, so that \cite[Proposition~5.1]{Ren:1980} gives us a time evolution $\sigma^{c_N}$ on $C^*(G_{\m,\Gamma}\ltimes \Omega_R^\m)$ such that 
\begin{equation}
\sigma^{c_N}_t(f)((n,k),w)=N(k)^{it}f((n,k),w)\quad\text{ for all } f\in C_c(G_{\m,\Gamma}\ltimes \Omega_R^\m) \text{ and } t\in\mathbb{R}.
\end{equation}

\begin{lemma}\label{lem:gpdmodel}
	Under the isomorphism $\vartheta:C_\lambda^*(P_{\m,\Gamma})\cong C^*(G_{\m,\Gamma}\ltimes \Omega_R^\m)$ from Proposition~\ref{prop:gpoidmodel}, the time evolution $\sigma$ from \eqref{eqn:timeevolution} is conjugated to $\sigma^{c_N}$, that is, $\sigma^{c_N}_t=\vartheta\circ\sigma_t\circ\vartheta^{-1}$ for every $t\in \mathbb{R}$.
\end{lemma}
\begin{proof}
We have $\vartheta(\lambda_{(b,a)})=1_{\{(b,a)\}\times \Omega_R^\m}$ for $(b,a)\in P_{\m,\Gamma}$, and a short calculation shows that $\sigma_t^{c_N}(1_{\{(b,a)\}\times \Omega_R^\m})=N(a)^{it}1_{\{(b,a)\}\times \Omega_R^\m}$ for all $t\in\mathbb{R}$. Thus, we have $\vartheta\circ\sigma_t\circ\vartheta^{-1}(1_{\{(b,a)\}\times \Omega_R^\m})=\sigma^{c_N}_t(1_{\{(b,a)\}\times \Omega_R^\m})$ for all $t\in\mathbb{R}$. Since the collection $\{\lambda_{(b,a)} : (b,a)\in P_{\m,\Gamma}\}$ generates $C_\lambda^*(P_{\m,\Gamma})$ as a C*-algebra, this is enough.
\end{proof}

Lemma~\ref{lem:gpdmodel} implies that there is an isomorphism of C*-dynamical systems 
\[
(C_\lambda^*(P_{\m,\Gamma}),\mathbb{R},\sigma)\cong (C^*(G_{\m,\Gamma}\ltimes \Omega_R^\m),\mathbb{R},\sigma^{c_N}),
\] 
so we may work with the latter system for our computations of KMS and ground states. From now on, we will write $\sigma$ rather than $\sigma^{c_N}$ for the time evolution on $C^*(G_{\m,\Gamma}\ltimes \Omega_R^\m)$.

Any state $\phi$ on $C^*(G_{\m,\Gamma}\ltimes\Omega_R^\m)$ defines a probability measure $\mu$ on $\Omega_R^\m$ by restricting $\phi$ to $C(\Omega_R^\m)$ and then applying the Riesz representation theorem to the state $\phi\vert_{C(\Omega_R^\m)}$. It is well-known, going back to \cite[Proposition~5.4]{Ren:1980}, that if $\phi$ is a $\sigma$-KMS$_\beta$ state, then the KMS$_\beta$ condition \eqref{KMScondition} forces the measure $\mu$ to be \emph{quasi-invariant with Radon-Nikodym cocycle given by $e^{-\beta c_N}=N^{-\beta}$}, that is, $\mu$ must satisfy the scaling condition
\begin{equation}\label{eqn:qi}
\mu((n,k)Z)=N(k)^{-\beta}\mu(Z) 
\end{equation}
for all $(n,k)\in G_{\m,\Gamma}$ and Borel sets $Z\subseteq \Omega_R^\m$ such that $(n,k)Z\subseteq \Omega_R^\m$. Moreover, the set of probability measures that satisfy \eqref{eqn:qi} forms a (possibly empty) Choquet simplex, see, for example, \cite[Exercise~3.3.1]{Ren:2009}.

There may be many $\sigma$-KMS$_\beta$ states on $C^*(G_{\m,\Gamma}\ltimes\Omega_R^\m)$ that define the same quasi-invariant measure on $\Omega_R^\m$, and \cite[Theorem~1.3]{Nesh:2013} gives a parameterization of all such $\sigma$-KMS$_\beta$ states in terms of traces on the C*-algebras of certain isotropy groups. Thus, to compute the $\sigma$-KMS$_\beta$ states on $C^*(G_{\m,\Gamma}\ltimes\Omega_R^\m)$ for $\beta<\infty$, we must first compute, for each fixed $\beta\in\mathbb{R}$, the simplex of all probability measures $\mu$ on $\Omega_R^\m$ that satisfy \eqref{eqn:qi}. It is easy to see that there are no such measures for $\beta<1$, as explained in Section~\ref{subsec:part(i)} below, and for this reason we restrict to the case $\beta\geq 1$ now.

\begin{lemma}\label{lem:spanningsets}
Let $\J:=\{(x+\a)\times\a^\times : x\in R, \a\in\I_\m^+\}$, and for $(x+\a)\times\a^\times\in\J$, let
\[
V_{(x+\a)\times\a^\times}:=\{[\mathbf{b},\bar{\mathbf{a}}]\in \Omega_R^\m : v_\p(\bar{\mathbf{a}})\geq v_\p(\a), v_\p(\mathbf{b}-x)\geq v_\p(\a) \text{ for all }\p\in\P_K^\m\}.
\]
If $\mu$ is a probability measure on $\Omega_R^\m$ and $\beta\geq 1$, then $\mu$ satisfies \eqref{eqn:qi} if and only if 
\begin{equation}\label{eqn:spanningsets}
\mu((n,k)V_X)=N(k)^{-\beta}\mu(V_X) 
\end{equation}
for all $(n,k)\in G_{\m,\Gamma}$ and $X\in\J$ such that $(n,k)X=\{(n+kx,ky) : (x,y)\in X\}$ lies in $\J$.

Moreover, for $\a\in\I_\m^+$, let 
\[
U_\a:=\a\hat{R}_S/\hat{R}_S^*=\{\bar{\mathbf{a}}\in\hat{R}_S/\hat{R}_S^* : v_\p(\bar{\mathbf{a}})\geq v_\p(\a) \text{ for all } \p\in\P_K^\m\}.
\]
Then for $k\in K_{\m,\Gamma}$ and $\a\in\I_\m^+$ such that $k\a\in\I_\m^+$, we have $k U_\a=U_{k\a}$, and a probability measure $\nu$ on $\hat{R}_S/\hat{R}_S^*$ satisfies 
\begin{equation}\label{eqn:qimultiplicative-1}
\nu(kZ)=N(k)^{-(\beta-1)}\nu(Z) 
\end{equation}
for every $k\in K_{\m,\Gamma}$ and every Borel set $Z\subseteq\hat{R}_S/\hat{R}_S^*$ such that $kZ\subseteq\hat{R}_S/\hat{R}_S^*$ if and only of 
\[
\nu(kU_\a)=N(k)^{-(\beta-1)}\nu(U_\a)
\]
for all $k\in K_{\m,\Gamma}$ and $\a\in\I_\m^+$ such that $k\a\in\I_\m^+$.
\end{lemma}
\begin{proof}
A calculation shows that $V_X$ is the support of the projection $\vartheta(E_X)\in C(\Omega_R^\m)$. The result follows from the fact that the projections $\{\vartheta(E_X) : X\in\J\}\cup\{0\}$ span a dense sub-*-algebra of $C(\Omega_R^\m)$.

A short calculation shows that for $k\in K_{\m,\Gamma}$ and $\a\in\I_\m^+$ such that $k\a\in\I_\m^+$, we have $k U_\a=U_{k\a}$. The last claim follows from the fact that the projections $\{1_{U_\a} : \a\in\I_\m^+\}$ span a dense sub-*-algebra of $C(\hat{R}_S/\hat{R}_S^*)$.
\end{proof}

Our next result is inspired by the proof of \cite[Proposition~2.1]{LN:2011} and \cite[Section~3]{Nesh:2013}. 

\begin{proposition}\label{prop:disintegration}
Let $\pi$ denote the quotient map $\hat{R}_S\times\hat{R}_S/\hat{R}_S^*\to\Omega_R^\m$, and let $m$ denote the normalized Haar measure on $\hat{R}_S$. Given a probability measure $\nu$ on $\hat{R}_S/\hat{R}_S^*$, form the product measure $m\times\nu$, and let $\pi_*(m\times\nu)$ denote the probability measure on $\Omega_R^\m$ obtained by pushing forward $m\times\nu$ under $\pi$.
For each fixed $\beta\geq 1$, the map $\nu\mapsto \pi_*(m\times\nu)$ defines an affine bijection from the set of probability measures $\nu$ on $\hat{R}_S/\hat{R}_S^*$ satisfying Equation~\eqref{eqn:qimultiplicative-1} onto the set of probability measures on $\Omega_R^\m$ satisfying \eqref{eqn:qi}. 
\end{proposition}
\begin{proof}
 Suppose that $\nu$ is a probability measure on $\hat{R}_S/\hat{R}_S^*$ satisfying \eqref{eqn:qimultiplicative-1}, and let $\mu:=\pi_*(m\times\nu)$. 
We need to show that $\mu$ satisfies \eqref{eqn:qi}. By Lemma~\ref{lem:spanningsets}, it suffices to show that $\mu$ satisfies \eqref{eqn:spanningsets}.

If $(n,k)\in G_{\m,\Gamma}$ and $X=(x+\a)\times\a^\times\in \J$ are such that $(n,k)X\in \J$, then $(n+kx+k\a\hat{R}_S)\times U_{k\a}=\pi^{-1}(V_{(n+kx+k\a)\times (ka)^\times})$, so we have
\[
\mu((n,k)V_X)=m(n+kx+k\a\hat{R}_S)\nu(U_{k\a})=N(k\a)^{-1}N(k)^{-(\beta-1)}\nu(U_\a)=N(k)^{-\beta}N(\a)^{-1}\nu(U_\a).
\]
For every $x\in R$,
\[
\mu(V_{(x+\a)\times\a^\times})=m(x+\a\hat{R}_S)\nu(U_\a)=N(\a)^{-1}\nu(U_\a),
\]
so we see that $\mu$ satisfies \eqref{eqn:spanningsets}. Since $\nu(U_\a)=N(\a)\mu(V_{(x+\a)\times\a^\times})$, and $\nu$ is determined by its values on the sets $U_\a$, we also conclude that the map $\nu\mapsto \pi_*(m\times\nu)$ is injective. 

It remains to check surjectivity. Suppose $\mu$ is a probability measure on $\Omega_R^\m$ satisfying \eqref{eqn:qi}, and let $q:\Omega_R^\m\to \hat{R}_S/\hat{R}_S^*$ denote the surjective map given by $q([\mathbf{b},\bar{\mathbf{a}}])=\bar{\mathbf{a}}$, so that $q\circ \pi=\pi_2$ is the projection from $\hat{R}_S\times\hat{R}_S/\hat{R}_S^*$ onto the second coordinate.
To show surjectivity, it is enough to show that
\begin{enumerate}
		\item $q_*\mu$ satisfies \eqref{eqn:qimultiplicative-1};
		\item $\mu=\pi_*(m\times q_*\mu)$. 
\end{enumerate} 
For $\a\in\I_\m^+$, let $V_{R\times\a^\times}:=\bigsqcup_{y\in R/\a}V_{(y+\a)\times\a^\times}$. Since $\mu$ satisfies \eqref{eqn:qi}, 
\begin{equation}\label{eqn:VR}
\mu(V_{R\times\a^\times})= \sum_{y\in R/\a}\mu(V_{(y+\a)\times\a^\times})=N(\a)\mu(V_{\a\times\a^\times}).
\end{equation}

Let $k\in K_{\m,\Gamma}$ and $\a\in\I_\m^+$ be such that $k\a\in\I_\m^+$. Then $q^{-1}(U_{k\a})=V_{R\times(k\a)^\times}$, so
\begin{align*}
q_*\mu(U_{k\a})=\mu(V_{R\times(k\a)^\times})&=N(k\a)\mu(V_{k\a\times(k\a)^\times})\\
&=N(k\a)N(k)^{-\beta}\mu(V_{\a\times\a^\times})\quad \text{(using that $\mu$ satisfies \eqref{eqn:qi})}\\
&=N(k)^{-(\beta-1)}\mu(V_{R\times\a^\times})\\
&=N(k)^{-(\beta-1)}q_*\mu(U_\a).
\end{align*}
Thus 1. holds. To show 2., it suffices to show that $\mu(V_{(x+\a)\times\a^\times})=\pi_*(m\times q_*\mu)(V_{(x+\a)\times\a^\times})$ for all $(x+\a)\times\a^\times\in\ \J$.
We have 
\[
\pi_*(m\times q_*\mu)(V_{(x+\a)\times\a^\times})=(m\times q_*\mu)((x+\a\hat{R}_S)\times U_\a)=N(\a)^{-1}q_*\mu(U_\a)=N(\a)^{-1}\mu(V_{R\times\a^\times}).
\]
Using \eqref{eqn:VR} and \eqref{eqn:qi}, we have $\mu(V_{R\times\a^\times})=N(\a)\mu(V_{\a\times\a^\times})=N(\a)\mu(V_{(x+\a)\times\a^\times})$. Hence, $\mu(V_{(x+\a)\times\a^\times})=\pi_*(m\times q_*\mu)(V_{(x+\a)\times\a^\times})$ for all $(x+\a)\times\a^\times\in\ \J$, as desired. 

It is not difficult to check that the map $\nu\mapsto m\times\nu$ is affine, and since the push-forward map $m\times\nu\mapsto\pi_*(m\times\nu)$ is also affine, we see that $\nu\mapsto\pi_*(m\times\nu)$ is affine.
\end{proof}

\subsection{The easy case: part (i)}\label{subsec:part(i)}

\begin{proof}[Proof of Theorem~\ref{thm:phasetransitions}(i)]
Suppose that $\phi$ is a $\sigma$-KMS$_\beta$ state on $C_\lambda^*(P_{\m,\Gamma})$. For each $x\in R$ and each $a\in R_{\m,\Gamma}$, the KMS$_\beta$ condition \eqref{KMScondition} yields
\[
\phi(\lambda_{(x,a)}\lambda_{(x,a)}^*)=N(a)^{-\beta}\phi(\lambda_{(x,a)}^*\lambda_{(x,a)})=N(a)^{-\beta}.
\]
Hence,
\[
0\leq \phi(1-\sum_{x\in R/aR}\lambda_{(x,a)}\lambda_{(x,a)}^*)=1-\sum_{x\in R/aR}\phi(\lambda_{(x,a)}\lambda_{(x,a)}^*)=1-N(a)^{1-\beta},
\]
so we must have $\beta \geq 1$. 
\end{proof}

\subsection{Uniqueness in the critical interval and the proof of part (ii)}\label{subsec:part(ii)}

Let $\nu_0:=\delta_{\bar{\mathbf{0}}}$ be the unit mass concentrated at the point $\bar{\mathbf{0}}\in\hat{R}_S/\hat{R}_S^*$. For each $\beta\in (0,\infty)$, let $\nu_{\beta,\p}$ be the probability measure on $\hat{ R}_\p/\hat{ R}_\p^*\cong\p^{\mathbb{N}\cup\{\infty\}}$ given by $\nu_{\beta,\p}=(1-N(\p)^{-\beta})\sum_{n=0}^\infty N(\p)^{-n\beta}\delta_{\p^n}$, and let $\nu_{\beta}:=\prod_{\p\in \P_K^\m}\nu_{\beta,\p}$. 

\begin{lemma}
For each $\beta\in [0,\infty)$, the measure $\nu_{\beta}$ satisfies 
\begin{equation}\label{eqn:qimultiplicative}
\nu(kZ)=N(k)^{-\beta}\nu(Z) 
\end{equation}
for every $k\in K_{\m,\Gamma}$ and every Borel set $Z\subseteq\hat{R}_S/\hat{R}_S^*$ such that $kZ\subseteq\hat{R}_S/\hat{R}_S^*$. 
Moreover, $\nu_\beta(U_\a)=N(\a)^{-\beta}$ for all $\a\in\I_\m^+$ where $U_\a=\a\hat{R}_S/\hat{R}_S^*$.
\end{lemma}
\begin{proof}
As pointed out in the proof of Proposition~\ref{prop:disintegration}, to show $\nu_\beta$ satisfies \eqref{eqn:qimultiplicative}, it suffices to show that $\nu_\beta$ satisfies
\[
\nu_\beta(kU_\a)=N(k)^{-\beta}\nu_\beta(U_\a)
\]
for all $k\in K_{\m,\Gamma}$ and $\a\in\I_\m^+$ such that $k\a\in\I_\m^+$.
Since $\bar{\mathbf{0}}\in U_\a$ for all $\a\in\I_\m^+$, it is easy to see that $\nu_0$ satisfies this condition. Now let $\beta\in (0,\infty)$. For any $\b\in\I_\m^+$, a calculation shows that $\nu_\beta(U_\b)=N(\b)^{-\beta}$ which settles the second claim. Using this, we have
\[
\nu_\beta(kU_\a)=\nu_\beta(U_{k\a})=N(k\a)^{-\beta}=N(k)^{-\beta}N(\a)^{-\beta}=N(k)^{-\beta}\nu_\beta(U_\a)
\]
as desired.
\end{proof}

The crux in computing the KMS$_\beta$ states for $\beta\in[1,2]$ is the following purely measure-theoretic result.

\begin{theorem}\label{thm:uniquemeasure}
For each $\beta\in [0,1]$, $\nu_\beta$ is the unique probability measure on $\hat{R}_S/\hat{R}_S^*$ satisfying \eqref{eqn:qimultiplicative}.
\end{theorem}

To prove Theorem~\ref{thm:uniquemeasure}, we will expand on an idea of Neshveyev's from the end of \cite[Section~3]{Nesh:2013}, which will put us in a setting where we can employ techniques analogous to those used for Bost--Connes type systems. 

We need two preliminary results. The first puts us in a situation where we can work with the lattice group $\I_\m$ of all fractional ideals coprime to $\m_0$, rather than the more complicated group $K_{\m,\Gamma}$. The following result is motivated by the general techniques from \cite{LN:2004} on extending KMS weights. 

\begin{lemma}\label{lem:1-1}
View $\hat{R}_S/\hat{R}_S^*$ as a subset of $\I_\m/i(K_{\m,\Gamma})\times \hat{R}_S/\hat{R}_S^*$ via the identification $\hat{R}_S/\hat{R}_S^*\simeq \{[R]\}\times \hat{R}_S/\hat{R}_S^*$. Then each probability measure $\nu$ on $\hat{R}_S/\hat{R}_S^*$ satisfying \eqref{eqn:qimultiplicative} has a unique extension to a finite measure $\tilde{\nu}$ on $\I_\m/i(K_{\m,\Gamma})\times \hat{R}_S/\hat{R}_S^*$ satisfying 
\begin{equation}\label{eqn:qimultiplicativetilde}
	\tilde{\nu}(\a Z)=N(\a)^{-\beta}\tilde{\nu}(Z)
\end{equation}
for all $\a\in \I_\m$ and Borel sets $Z\subseteq \I_\m/i(K_{\m,\Gamma})\times \hat{R}_S/\hat{R}_S^*$ such that $\a Z\subseteq \I_\m/i(K_{\m,\Gamma})\times \hat{R}_S/\hat{R}_S^*$ where $\a Z=\{(\a\k,\a\bar{\mathbf{a}}) : (\k,\bar{\mathbf{a}})\in Z\}$.
\end{lemma}
\begin{proof}
Our proof is similar to that of \cite[Lemma~2.2]{LLN:2007}. For $\a\in\I_\m$, let $[\a]$ denote the class of $\a$ in $\I_\m/i(K_{\m,\Gamma})$, and for each integral ideal $\a$, let $Y_\a:= \{[R]\}\times U_\a$, so that $\hat{R}_S/\hat{R}_{S}^*\simeq Y_R\subseteq\I_\m/i(K_{\m,\Gamma})\times \hat{R}_S/\hat{R}_S^*$.

Suppose that $\nu$ is a probability measure on $\hat{R}_S/\hat{R}_S^*$ satisfying \eqref{eqn:qimultiplicative}. We first show that there can be at most one measure $\mu$ on $\I_\m/i(K_{\m,\Gamma})\times \hat{R}_S/\hat{R}_S^*$ that both satisfies \eqref{eqn:qimultiplicativetilde} and extends $\nu$. Indeed, suppose that $\mu$ is such a measure, and for each class $\k\in \I_\m/i(K_{\m,\Gamma})$, choose an integral ideal $\a_\k\in\k$; for $\k=[R]$, take $\a_\k=R$. Then
\[
	\I_\m/i(K_{\m,\Gamma})\times \hat{R}_S/\hat{R}_S^*=\bigsqcup_{\k\in \I_\m/i(K_{\m,\Gamma})}\a_\k^{-1}Y_{\a_\k},
\]
so, for any Borel set $Z$,
\[
	\mu(Z)=\sum_{\k\in \I_\m/i(K_{\m,\Gamma})}\mu(Z\cap \a_\k^{-1}Y_{\a_\k}).
\]
Since $\mu$ satisfies \eqref{eqn:qimultiplicativetilde}, $\mu(Z\cap \a_\k^{-1}Y_{\a_\k})=\mu(\a_\k^{-1}(\a_\k Z\cap Y_{\a_\k}))=N(\a_\k)^\beta\mu(\a_\k Z\cap Y_{\a_\k})$. Since $Y_{\a_\k}\subseteq Y_{R}$, we see that $\mu$ is determined by its restriction to $Y_R$.
Thus, there can be at most one measure on $\I_\m/i(K_{\m,\Gamma})\times \hat{R}_S/\hat{R}_S^*$ that both satisfies \eqref{eqn:qimultiplicativetilde} and extends $\nu$. We now proceed to construct this extension.
Define $\tilde{\nu}$ on $\I_\m/i(K_{\m,\Gamma})\times \hat{R}_S/\hat{R}_S^*$ by
	\[
	\tilde{\nu}(Z)=\sum_{\k\in\I_\m/i(K_{\m,\Gamma})}N(\a_\k)^\beta\nu(\a_\k Z\cap Y_{\a_\k})
	\]
for Borel sets $Z\subseteq \I_\m/i(K_{\m,\Gamma})\times \hat{R}_S/\hat{R}_S^*$. A short calculation shows that $\tilde{\nu}$ is a finite measure extending $\nu$. We need to show that $\tilde{\nu}$ satisfies \eqref{eqn:qimultiplicativetilde}. For each $\k$, let $b_\k\in K_{\m,\Gamma}$ be such that $\a\a_\k=b_\k\a_{\a\cdot\k}$. We have
	
\begin{align*}
	\tilde{\nu}(\a Z)&=\sum_{\k\in\I_\m/i(K_{\m,\Gamma})}N(\a_\k)^\beta\nu(\a_\k\a Z\cap Y_{\a_\k})\\
	&=N(\a)^{-\beta}\sum_{\k\in\I_\m/i(K_{\m,\Gamma})}N(\a_\k\a)^\beta\nu(\a_\k\a Z\cap Y_{\a_\k})\\
	&=N(\a)^{-\beta}\sum_{\k\in\I_\m/i(K_{\m,\Gamma})}N(\a_\k\a)^\beta\nu(\a_\k\a Z\cap Y_{\a\a_\k})\quad (\text{since }a_\k(\a Z\cap Y_{R})=\a_\k(\a Z\cap Y_{\a})) \\
	&=N(\a)^{-\beta}\sum_{\k\in\I_\m/i(K_{\m,\Gamma})}N(b_\k\a_{\a\cdot\k})^\beta\nu(b_\k\a_{\a\cdot\k} Z\cap Y_{b_\k\a_{\a\cdot\k}})\\
	&=N(\a)^{-\beta}\sum_{\k\in\I_\m/i(K_{\m,\Gamma})}N(b_\k\a_{\a\cdot\k})^\beta N(b_\k)^{-\beta}\nu(\a_{\a\cdot\k} Z\cap Y_{\a_{\a\cdot\k}})\quad (\text{using \eqref{eqn:qimultiplicative}})\\
	&=N(\a)^{-\beta}\sum_{\tilde{\k}\in\I_\m/i(K_{\m,\Gamma})}N(\a_{\tilde{\k}})^\beta\nu(\a_{\tilde{\k}} Z\cap Y_{\a_{\tilde{\k}}})\\
	&=N(\a)^{-\beta}\tilde{\nu}(Z).
\end{align*}
This concludes the proof.
\end{proof}

The following ergodicity results is the key step towards Theorem~\ref{thm:uniquemeasure}.

\begin{proposition}\label{prop:ergodicity}
Let $\beta\in (0,1]$ and suppose that $\nu$ is a probability measure on $\I_\m/i(K_{\m,\Gamma})\times \hat{R}_S/\hat{R}_S^*$ satisfying \eqref{eqn:qimultiplicativetilde}. 
Then the closed subspace 
	\[
	H=\{f\in L^2(\I_\m/i(K_{\m,\Gamma})\times \hat{R}_S/\hat{R}_S^*,\nu) : f(\a z)=f(z) \text{ for } \a\in \I_\m^+, z\in \I_\m/i(K_{\m,\Gamma})\times \hat{R}_S/\hat{R}_S^*\}
	\] 
of $L^2(\I_\m/i(K_{\m,\Gamma})\times \hat{R}_S/\hat{R}_S^*,\nu)$ consisting of $\I_\m^+$-invariant functions coincides with the constant functions. That is, the partial action of $\I_\m$ on $(\I_\m/i(K_{\m,\Gamma})\times \hat{R}_S/\hat{R}_S^*,\nu)$ is ergodic.
\end{proposition}
\begin{proof}
The proof is similar to that of \cite[Theorem~2.1(ii)]{LLN:2009}. 
Let $P$ be the orthogonal projection from $L^2(\I_\m/i(K_{\m,\Gamma})\times \hat{R}_S/\hat{R}_S^*,\nu)$ onto $H$; we need to show that $Pf$ is a constant function for every $f$.
For this, it suffices to compute $P$ at pull-backs of functions on
\[
\I_\m/i(K_{\m,\Gamma})\times \prod_{\p\in F}\p^{\mathbb{N}\cup\{\infty\}}
\]
for every non-empty finite subset $F\subseteq\P_K^\m$. Now fix such an $F$, and let $\I_{F^c}^+$ be the free submonoid of $\I_\m^+$ generated by the primes in $F$. Up to a set of measure zero, $\I_\m/i(K_{\m,\Gamma})\times \prod_{\p\in F}\p^{\mathbb{N}\cup\{\infty\}}$ coincides with
\[
\bigsqcup_{\a\in \I_{F^c}^+}\a(\I_\m/i(K_{\m,\Gamma})\times\{\underbrace{(1,...,1)}_{|F|\text{-times}}\}).
\] 
Since $\I_\m/i(K_{\m,\Gamma})\times\{(1,...,1)\}\simeq \I_\m/i(K_{\m,\Gamma})$ is a finite group, it suffices to compute, for each fixed $\a\in\I_{F^c}^+$ and character $\tilde{\chi}$ of $\I_\m/i(K_{\m,\Gamma})$, $Pf$ where $f$ is the pull-back of
	\[
	\I_\m/i(K_{\m,\Gamma})\times\prod_{\p\in F}\p^{\mathbb{N}\cup\{\infty\}} \ni (\k,\bar{\mathbf{a}}) \mapsto 
	\begin{cases}
	\tilde{\chi}([\a]^{-1}\k) & \text{ if } (\k,\bar{\mathbf{a}})\in \a(\I_\m/i(K_{\m,\Gamma})\times\{\underbrace{(1,...,1)}_{|F|\text{-times}}\})\\
	0 & \text{ otherwise}.
	\end{cases}
	\]
The character $\chi:\I_\m\to \mathbb{T}$ defined by $\chi(\a):=\tilde{\chi}([\a])$ satisfies $\chi(i(K_{\m,1}))=\{1\}$, and is thus a (generalized) Dirichlet character modulo $\m$.\\
For each finite subset $\tilde{F}\subseteq\P_K^\m$, let $P_{\tilde{F}}$ be the orthogonal projection onto the subspace $H_{\tilde{F}}$ consisting of $\I_{\tilde{F}^c}^+$-invariant functions, so that the projection $P$ is the decreasing strong operator limit of the net $(P_{\tilde{F}})_{\tilde{F}}$. Also let 
	\[
	W_{\tilde{F}}:=\I_\m/i(K_{\m,\Gamma})\times\{\underbrace{(1,...,1)}_{|\tilde{F}|\text{-times}}\}\times\prod_{\p\in \tilde{F}^c}\p^{\mathbb{N}\cup\{\infty\}},
	\]
so that the sets $\b W_{\tilde{F}}$ are disjoint for $\b\in\I_{\tilde{F}^c}^+$, and their union has full $\nu$-measure.
Applying \cite[Proposition~1.2(2)]{LLN:2009} with, in the notation from the statement of \cite[Proposition~1.2]{LLN:2009}, $Y_0=W_{\tilde{F}}$, $Y=\I_\m/i(K_{\m,\Gamma})\times\prod_{\p\in \P_K^\m}\p^{\mathbb{N}\cup\{\infty\}}$, $S=\I_{\tilde{F}^c}^+$ and $G=\I_{\tilde{F}^c}$, we get that the projection $P_{\tilde{F}}$ is given explicitly by
	\begin{equation*}
	P_{\tilde{F}}f\vert_{\I_{\tilde{F}^c}^+w}\equiv \frac{1}{\zeta_{\tilde{F}}(\beta)}\sum_{\b\in \I_{\tilde{F}^c}^+}N(\b)^{-\beta}f(\b w)
	\end{equation*}
for $w=(\k,\bar{\mathbf{a}})\in W_{\tilde{F}}$ where $\zeta_{\tilde{F}}(\beta):=\prod_{\p\in \tilde{F}}(1-N(\p)^{-\beta})^{-1}=\sum_{\b\in \I_{\tilde{F}^c}^+}N(\b)^{-\beta}$.
Now suppose that $\tilde{F}\supseteq F$. Then for $f(\b w)$ to be non-zero, it is necessary that $\b\in \a \I_{(\tilde{F}\setminus F)^c}^+$, and, in this case, $f(\b w)=\tilde{\chi}([\a^{-1}\b]\k)=\chi(\a^{-1}\b)\tilde{\chi}(\k)$. Hence, for $w=(\k,\bar{\mathbf{a}})\in W_{\tilde{F}}$, we have
	\begin{equation}\label{eqn:proj}
		P_{\tilde{F}}f\vert_{\I_{\tilde{F}^c}^+w}=\frac{1}{\zeta_{\tilde{F}}(\beta)}\sum_{\b\in \a \I_{(\tilde{F}\setminus F)^c}^+}N(\b)^{-\beta}\chi(\a^{-1}\b)\tilde{\chi}(\k)=\frac{N(\a)^{-\beta}\tilde{\chi}(\k)}{\zeta_{\tilde{F}}(\beta)}\sum_{\c\in \I_{(\tilde{F}\setminus F)^c}^+}N(\c)^{-\beta}\chi(\c).
	\end{equation}
If $\chi$ is the trivial character, then the right-hand side of \eqref{eqn:proj} equals 
	\[
	N(\a)^{-\beta}\frac{\prod_{\p\in \tilde{F}\setminus F}(1-N(\p)^{-\beta})^{-1}}{\prod_{\p\in\tilde{F}}(1-N(\p)^{-\beta})^{-1}}=N(\a)^{-\beta}\prod_{\p\in F}(1-N(\p)^{-\beta}),
	\]
	so $Pf=\lim_{\tilde{F}}P_{\tilde{F}}f$ is constant. Now suppose that $\chi$ is non-trivial. Then,
	\begin{align*}
	||P_{\tilde{F}}f||_{L^2(\nu)}^2&=\sum_{\b\in \I_{\tilde{F}^c}^+}\int_{\b W_{\tilde{F}}}|P_{\tilde{F}}f|^2\;d\nu\\
	&=\zeta_{\tilde{F}}(\beta)\int_{W_{\tilde{F}}}|P_{\tilde{F}}f|^2\;d\nu\quad\text{(using that $\nu$ satisfies \eqref{eqn:qimultiplicativetilde})}\\
	&=\zeta_{\tilde{F}}(\beta) \left|\frac{N(\a)^{-\beta}}{\zeta_{\tilde{F}}(\beta)}\sum_{\c\in \I_{(\tilde{F}\setminus F)^c}^+}N(\c)^{-\beta}\chi(\c)\right|^2\nu(W_{\tilde{F}})\quad \text{(using \eqref{eqn:proj})}\\
	&=\left(N(\a)^{-\beta}\frac{\prod_{\p\in \tilde{F}}|1-N(\p)^{-\beta}|}{\prod_{\p\in \tilde{F}\setminus F}|1-\chi(\p)N(\p)^{-\beta}|}\right)^2 \quad \text{(since $\nu(W_{\tilde{F}})=\zeta_{\tilde{F}}(\beta)^{-1}$).}
	\end{align*}
Hence,
	\begin{equation}\label{eqn:normestimate}
	||Pf||_{L^2(\nu)}=\lim_{\tilde{F}}||P_{\tilde{F}}f||_{L^2(\nu)}=N(\a)^{-\beta}\lim_{\tilde{F}}\frac{\prod_{\p\in \tilde{F}}|1-N(\p)^{-\beta}|}{\prod_{\p\in \tilde{F}\setminus F}|1-\chi(\p)N(\p)^{-\beta}|}.
	\end{equation}
For each $\tilde{F}$, the function
	\[
	\beta\mapsto \frac{\prod_{\p\in \tilde{F}}|1-N(\p)^{-\beta}|}{\prod_{\p\in \tilde{F}\setminus F}|1-\chi(\p)N(\p)^{-\beta}|}
	\]
	is increasing on $(0,\infty)$ since for $\p\in\tilde{F}\setminus F$, the function $\beta\mapsto \frac{|1-N(\p)^{-\beta}|}{|1-\chi(\p)N(\p)^{-\beta}|}$ is increasing on $(0,\infty)$. For $\beta>1$, the limit $\lim_{\tilde{F}}\frac{\prod_{\p\in \tilde{F}}|1-N(\p)^{-\beta}|}{\prod_{\p\in \tilde{F}\setminus F}|1-\chi(\p)N(\p)^{-\beta}|}$ exists and is equal to 
	\[
	\frac{|L(\chi,\beta)|}{\zeta_K(\beta)}\frac{\prod_{\p\in F\cup S}|(1-\chi(\p)N(\p)^{-\beta})|}{\prod_{\p\in S}(1-N(\p)^{-\beta})}
	\]
	where $L(\chi,\beta)$ is the (generalized) Dirichlet $L$-function associated with $\chi$ and $\zeta_K(\beta)$ is the Dedekind zeta function of $K$. Now as $\beta\to1^+$, $L(\chi,\beta)$ tends to a finite value, see, for example, \cite[Chapter~VI,~Corollary~2.11]{Mil:2013}, whereas $\zeta_K(\beta)$ has a pole at $\beta=1$ by \cite[Chapter~VI,~Corollary~2.12]{Mil:2013}. Therefore, the right hand side of \eqref{eqn:normestimate} converges to zero for all $\beta\in(0,1]$, so $||Pf||_{L^2(\nu)}=0$. In particular, $Pf$ is constant.
\end{proof}

We are now ready to prove Theorem~\ref{thm:uniquemeasure}.

\begin{proof}[Proof of Theorem~\ref{thm:uniquemeasure}]
We first deal with the case $\beta=0$. Suppose $\nu$ is a $K_{\m,\Gamma}$-invariant probability measure on $\hat{R}_S/\hat{R}_S^*$. Then, in particular, we have $\nu(a \hat{R}_S/\hat{R}_S^*)=1$ for every $a\in  R_{\m,\Gamma}$, which implies that $\nu(\bigcap_a a \hat{R}_S/\hat{R}_S^*)=1$. Since $\bigcap_a a \hat{R}_S/\hat{R}_S^*=\{\bar{\mathbf{0}}\}$, we have $\nu=\delta_{\bar{\mathbf{0}}}$, as desired.
	
Now let $\beta\in (0,1]$. By Lemma~\ref{lem:1-1}, it suffices to show that the probability measure $\bar{\nu}_\beta$ on $\I_\m/i(K_{\m,\Gamma})\times \hat{R}_S/\hat{R}_S^*$ obtained by normalizing $\tilde{\nu}_\beta$ is the unique probability measure satisfying \eqref{eqn:qimultiplicativetilde}.
The set of probability measures that satisfy \eqref{eqn:qimultiplicativetilde} forms a simplex $\Sigma$, and Proposition~\ref{prop:ergodicity} says that all measures in $\Sigma$ are ergodic. A non-trivial convex combination of measures is never ergodic, so we have $\Sigma=\{\bar{\nu}_\beta\}$.
\end{proof}

We are now ready for the proof of uniqueness for $\beta\in[1,2]$.

\begin{proof}[Proof of the existence and uniqueness statement in Theorem~\ref{thm:phasetransitions}(ii)]
Let $\pi$ denote the quotient map $\hat{R}_S\times \hat{R}_S/\hat{R}_S^*\to\Omega_R^\m$ and $m$ the normalized Haar measure on $\hat{R}_S$.
For $\beta\in [1,2]$, let $\mu_\beta:=\pi_*(m\times \nu_{\beta-1})$ be push-forward of the product measure $m\times \nu_{\beta-1}$ under $\pi$. It follows from  Theorem~\ref{thm:uniquemeasure} combined with Proposition~\ref{prop:disintegration} that $\mu_\beta$ is the unique probability measure on $\Omega_R^\m$ satisfying \eqref{eqn:qi}. 

We now show that the set of points in $\Omega_R^\m$ with non-trivial isotropy has $\mu_\beta$-measure zero. Our proof is almost the same as that of \cite[Lemma~3.3]{Nesh:2013}, but we include it for completeness.
For $\beta=1$, we can identify the measure space $(\Omega_R^\m,\mu_1)$ with $(\hat{R}_S,m)$, and the partial action of $G_{\m,\Gamma}$ on $(\hat{R}_S,m)$ is given by the usual ``$ax+b$'' action, that is, by $(n,k)\mathbf{a}=n+k\mathbf{a}$ for $(n,k)\in G_{\m,\Gamma}$ and $\mathbf{a}\in\hat{R}_S$ such that $n+k\mathbf{a}\in\hat{R}_S$. In this case, each non-identity element of $G_{\m,\Gamma}$ has at most one fixed point in $\hat{R}_S$, and every point in $\hat{R}_S$ has $m$-measure zero. It follows that the set of points in $\Omega_R^\m$ with non-trivial isotropy has $\mu_1$-measure zero.

Now let $\beta\in (1,2]$. To show that the set of points in $\Omega_R^\m$ with non-trivial isotropy has $\mu_\beta$-measure zero, it suffices to show that for each non-trivial element $\gamma=(n,k)\in G_{\m,\Gamma}$, the set of points in $[\mathbf{b},\bar{\mathbf{a}}]\in\Omega_R^\m$ fixed by $\gamma$ has $\mu_\beta$-measure zero. As was done in a similar situation in the proof of \cite[Proposition~2.1]{LN:2011}, we can disintegrate $\mu_\beta$ with respect to the canonical projection map $\Omega_R^\m\to \hat{R}_S/\hat{R}_S^*$ to get that
\[
\int_{\Omega_R^\m} f([\mathbf{b},\bar{\mathbf{a}}])\;d\mu_\beta([\mathbf{b},\bar{\mathbf{a}}])=\int_{\hat{R}_S/\hat{R}_S^*}\left(\int_{\hat{R}_S/\bar{\mathbf{a}}\hat{R}_S}f([\mathbf{b},\bar{\mathbf{a}}]) \;d\lambda_{\bar{\mathbf{a}}}(\dot{\mathbf{b}})\right)\;d\nu_\beta(\bar{\mathbf{a}})
\]
for each $f\in C(\Omega_R^\m)$ where for $\nu_\beta$-a.e. $\bar{\mathbf{a}}\in\hat{R}_S/\hat{R}_S^*$, the probability measure $\lambda_{\bar{\mathbf{a}}}$ is equal to the normalized Haar measure $m_{\bar{\mathbf{a}}}$ on the quotient $\hat{R}_S/\bar{\mathbf{a}}\hat{R}_S$ and where $\dot{\mathbf{b}}$ denotes the image of $\mathbf{b}$ under the quotient map $\hat{R}_S\to\hat{R}_S/\bar{\mathbf{a}}\hat{R}_S$.
Hence, to show that set of points in $[\mathbf{b},\bar{\mathbf{a}}]\in\Omega_R^\m$ fixed by $\gamma$ has $\mu_\beta$-measure zero, it is enough to show that for $\nu_\beta$-a.e. $\bar{\mathbf{a}}\in\hat{R}_S/\hat{R}_S^*$, the set $A_{\gamma,\bar{\mathbf{a}}}:=\{\dot{\mathbf{b}}\in \hat{R}_S/\bar{\mathbf{a}}\hat{R}_S : \gamma[\mathbf{b},\mathbf{a}]=[\mathbf{b},\mathbf{a}]\}$ has $\m_{\bar{\mathbf{a}}}$-measure zero. 
Let 
\[
\AA_S^*:=\Big\{\mathbf{a}=(\mathbf{a}_\p)_\p\in\prod_{\p\in\P_K\setminus S}K_\p^* :  \mathbf{a}_\p\in R_\p^* \text{ for all but finitely many } \p\Big\},
\]
so that $(\AA_S^*\cap\hat{R}_S)/\hat{R}_S^*$ can be identified with the countable set $\I_\m^+$. Since $\beta\in(1,2]$, the set 
\[
\{\bar{\mathbf{a}}\in \hat{R}_S/\hat{R}_S^* : \bar{\mathbf{a}} \notin (\AA_S^*\cap\hat{R}_S)/\hat{R}_S^*\text{ and }\bar{\mathbf{a}}_\p\neq 0 \text{ for all } \p\}
\] 
has full $\nu_\beta$-measure, so we only need to show that $m_{\bar{\mathbf{a}}}(A_{\gamma,\bar{\mathbf{a}}})=0$ for all $\bar{\mathbf{a}} \notin (\AA_S^*\cap\hat{R}_S)/\hat{R}_S^*$ such that $\bar{\mathbf{a}}_\p\neq 0$ for all $\p$.
The set $A_{\gamma,\bar{\mathbf{a}}}$ is empty unless $k\bar{\mathbf{a}}=\bar{\mathbf{a}}$, in which case the condition $\bar{\mathbf{a}}_\p\neq 0$ for all $\p$ forces $k\in R_{\m,\Gamma}^*$. Now we see that $A_{\gamma,\bar{\mathbf{a}}}=\{\dot{\mathbf{b}}\in\hat{R}_S/\bar{\mathbf{a}}\hat{R}_S: (k-1)\dot{\mathbf{b}}=\dot{\mathbf{b}}\}$. If $k=1$, then $A_{\gamma,\bar{\mathbf{a}}}$ is empty unless $n\in\bar{\mathbf{a}}\hat{R}_S$; in this case, we must have $n=0$ because of the assumption that $\bar{\mathbf{a}} \notin (\AA_S^*\cap\hat{R}_S)/\hat{R}_S^*$, which then implies $\gamma=(0,1)$. Since we assumed $\gamma$ to be non-trivial, we see that $A_{\gamma,\bar{\mathbf{a}}}$ can be non-empty only when $k\neq 1$. In this case, $k-1$ lies in $R_\p^*$ for all $\p$ outside some finite set $F\subseteq\P_K^\m$, and thus any $\dot{\mathbf{c}}\in A_{\gamma,\bar{\mathbf{a}}}$ is uniquely determined for all $\p\in\P_K^\m\setminus F$. We can now obtain the inequality 
\[
m_{\bar{\mathbf{a}}}(A_{\gamma,\bar{\mathbf{a}}})\leq \prod_{\p\in \P_K^\m\setminus F}|R_\p/\bar{\mathbf{a}}_\p R_\p|^{-1},
\]
and the product on the right hand side diverges to zero because of the assumption that there are infinitely many $\p$ for which $\bar{\mathbf{a}}_\p$ does not lie in $R_\p^*$. This finishes our proof that the set of points in $\Omega_R^\m$ with non-trivial isotropy has $\mu_\beta$-measure zero.

Now \cite[Theorem~1.3]{Nesh:2013} implies that $\mu_\beta\circ \E$ is the unique $\sigma$-KMS$_\beta$ state on $C^*(G_{\m,\Gamma}\ltimes\Omega_R^\m)$ for $\beta\in [1,2]$. Moreover, since $\mu_\beta(V_{(x+\a)\times\a^\times})=N(\a)^{-\beta}$, we are done.
\end{proof}

\begin{remark}
 If $\m$ and $\Gamma$ are trivial, or if $\m=\m_\infty$ consists of all the real embeddings of $K$ and $\Gamma$ is trivial, then Theorem~\ref{thm:uniquemeasure} can be deduced from \cite[Theorem~3.1]{Nesh:2013}. However, even in this case, our proof here is different: in \cite[Section~3]{Nesh:2013}, the special case of Theorem~\ref{thm:uniquemeasure} is obtained by using known results from \cite{LNT:2013} for the Hecke C*-dynamical system associated with the Hecke pair $(K\rtimes K_+^*,R\rtimes R_+^*)$, whereas we give a more direct proof.
\end{remark}

\subsection{Low temperature KMS states: the proof of part (iii)}\label{subsec:part(iii)}
The map $\a\mapsto \prod_{\p\in\P_K^\m}\p^{v_\p(\a)}$ canonically identifies $\I_\m^+$ with a subset of $\prod_{\p\in \P_K^\m}\p^{\mathbb{N}\cup\{\infty\}}$. Composing with the canonical homeomorphism $\prod_{\p\in\P_K^\m}\p^{\mathbb{N}\cup\{\infty\}}\simeq\hat{R}_S/\hat{R}_S^*$, we may view $\I_\m^+$ as a subset of $\hat{R}_S/\hat{R}_S^*$.
The image of $\I_\m^+$ in $\hat{R}_S/\hat{R}_S^*$ consists of those ``super ideals'' that are coprime to $\m_0$ and have only finitely many divisors, that is, with the set 
\[
\{\bar{\mathbf{a}}\in \hat{R}_S/\hat{R}_S^* : \bar{\mathbf{a}}\in U_\a \text{ for only finitely many } \a\}
\]
where $U_\a=\a\hat{R}_S/\hat{R}_S^*$. We will show that for each $\beta>1$, every probability measure on $\hat{R}_S/\hat{R}_S^*$ that satisfies \eqref{eqn:qimultiplicative} must be concentrated on this countable set, and thus is a convex combination of measures that are concentrated on orbits for the partial action of $K_{\m,\Gamma}$ on $\hat{R}_S/\hat{R}_S^*$. These orbits are precisely the sets $\k\cap \I_\m^+$ for $\k\in\I_\m/i(K_{\m,\Gamma})$.

The \emph{partial zeta function} associated with a class $\k\in\I_\m/i(K_{\m,\Gamma})$ is the Dirichlet series 
\[
\zeta_\k(s):=\sum_{\a\in \k\cap\I_\m^+}N(\a)^{-s},
\]
which converges for all complex numbers $s$ with real part greater than $1$. 

\begin{lemma}\label{lem:qimeasuresforlargebeta}
For each $\beta\in (1,\infty)$ and each class $\k\in\I_\m/i(K_{\m,\Gamma})$, let $\nu_{\beta,\k}$ be the probability measure on $\hat{R}_S/\hat{R}_S^*$ given by
\[
\nu_{\beta,\k}:=\frac{1}{\zeta_\k(\beta)}\sum_{\a\in \k\cap \I_\m^+}N(\a)^{-\beta}\delta_{\a}
\]
where $\delta_\a$ denotes the unit mass concentrated at the point $\a\in\hat{R}_S/\hat{R}_S^*$. Then each measure $\nu_{\beta,\k}$ satisfies \eqref{eqn:qimultiplicative}. Moreover, any probability measure $\nu$ that satisfies \eqref{eqn:qimultiplicative} for $\beta\in (1,\infty)$ is a convex combination of measures from $\{\nu_{\beta,\k} : \k\in \I_\m/i(K_{\m,\Gamma})\}$.
\end{lemma}
\begin{proof}
For $\beta\in (1,\infty)$ and $\k\in \I_\m/i(K_{\m,\Gamma})$, a calculation shows that the measure $\nu_{\beta,\k}$ satisfies \eqref{eqn:qimultiplicative}.

Now fix $\beta\in(1,\infty)$, and let $\nu$ be a probability measure on $\hat{R}_S/\hat{R}_S^*$ that satisfies \eqref{eqn:qimultiplicative}. Recall that the inverse of a fractional ideal $\a$ in $\I_\m$ is given by $\a^{-1}:=\{x\in K : x\a\subseteq R\}$. For $\a\in\I_\m^+$ and $x\in \a^{-1}\cap K_{\m,\Gamma}$, we have $xU_\a=U_{x\a}$. Now,
\begin{align*}
\sum_{\a\in\k\cap \I_\m^+}\nu(U_\a)&=\sum_{x\in (\a_\k^{-1}\cap K_{\m,\Gamma})/R_{\m,\Gamma}^*}\nu(U_{x\a_\k})\\
&=\sum_{x\in (\a_\k^{-1}\cap K_{\m,\Gamma})/R_{\m,\Gamma}^*}N(x)^{-\beta}\nu(U_{\a_\k})\quad\text{(using \eqref{eqn:qimultiplicative})}\\
&=\sum_{x\in (\a_\k^{-1}\cap K_{\m,\Gamma})/R_{\m,\Gamma}^*}N(x\a_\k)^{-\beta}N(\a_\k)^{\beta}\nu(U_{\a_\k})\\
&=\zeta_\k(\beta)N(\a_\k)^{\beta}\nu(U_{\a_\k}).
\end{align*}
Thus, since $\beta>1$, 
	\[
	\sum_{\a\in \I_\m^+}\nu(U_\a)=\sum_{\k\in \I_\m/i(K_{\m,\Gamma})}\zeta_\k(\beta)N(\a_\k)^{\beta}\nu(U_{\a_\k})<\infty,
	\]
so the Borel-Cantelli lemma implies that $\nu$ is concentrated on the set 
	\[
	\{\bar{\mathbf{a}}\in \hat{R}_S/\hat{R}_S^*: \bar{\mathbf{a}}\in U_\a \text{ for only finitely many }\a\}.
	\]
This set coincides with the canonical copy of $\I_\m^+$ in $\hat{R}_S/\hat{R}_S^*$. Since $\nu$ satisfies \eqref{eqn:qimultiplicative} and $\I_\m^+$ is countable, the set of points that have positive $\nu$-measure must be a (disjoint) union of orbits for the partial action of $K_{\m,\Gamma}$ on $\I_\m^+$, and $\nu$ is a convex combination of its normalized restrictions to these orbits; moreover, these orbits are precisely the sets $\k\cap \I_\m^+$ for $\k\in\I_\m/i(K_{\m,\Gamma})$, and a calculation shows that $\nu_{\beta,\k}$ is the only probability measure that both satisfies \eqref{eqn:qimultiplicative} and is concentrated on $\k\cap \I_\m^+$, so we are done.
\end{proof}

We are now ready for the proof of Theorem~\ref{thm:phasetransitions}(iii).

\begin{proof}[Proof of Theorem~\ref{thm:phasetransitions}(iii)]
As before, let $\pi$ denote the quotient map $\hat{R}_S\times \hat{R}_S/\hat{R}_S^*\to\Omega_R^\m$, and let $m$ be the normalized Haar measure on $\hat{R}_S$.
For each $\beta>2$ and each class $\k\in\I_\m/i(K_{\m,\Gamma})$, let $\mu_{\beta,\k}:=\pi_*(m\times\nu_{\beta-1,\k})$ be the push-forward of the product measure $m\times\nu_{\beta-1,\k}$ under $\pi$.
By Proposition~\ref{prop:disintegration}, the map $\nu_{\beta-1}\mapsto \pi_*(m\times\nu_{\beta-1,\k})$ establishes an affine bijection from the simplex of probability measures on $\hat{R}_S/\hat{R}_S^*$ that satisfy \eqref{eqn:qimultiplicative-1} onto the simplex of probability measures on $\Omega_R^\m$ that satisfy \eqref{eqn:qi}; hence, by Lemma~\ref{lem:qimeasuresforlargebeta}, every probability measure on $\Omega_R^\m$ that satisfies \eqref{eqn:qi} is a convex combination of measures from the set $\{\pi_*(m\times\nu_{\beta-1,\k}): \k\in \I_\m/i(K_{\m,\Gamma})\}$.

For each $\a\in \I_\m^+$, there are exactly $N(\a)$ points $[\mathbf{b},\a]$ in $\Omega_R^\m$ with second component equal to $\a$. Indeed, we can always write $[\mathbf{b},\a]=[x,\a]$ for some $x\in\hat{R}_S/\a\hat{R}_S\cong R/\a$ with $\mathbf{b}-x\in \a\hat{R}_S$. 
Hence, for each $\k\in \I_\m/i(K_{\m,\Gamma})$, the set $\{[b,\a]\in \Omega_R^\m : \a\in\k\}$ is countable. 
Moreover, the partial action of $G_{\m,\Gamma}$ on $\{[b,\a]\in \Omega_R^\m : \a\in\k\}$ is transitive, and the measure $\mu_{\beta,\k}$ is concentrated on $\{[b,\a]\in \Omega_R^\m : \a\in\k\}$. 
This set contains the point $[0,\a_{\k,1}]$, which has isotropy group $\a_{\k,1}\rtimes R_{\m,\Gamma}^*$, and the $\sigma$-KMS$_\beta$ states $\phi$ satisfying $\phi\vert_{C(\Omega_R^\m)}=\mu_{\beta,\k}$ are in one-to-one correspondence with the tracial states of the group C*-algebra $C^*(\a_{\k,1}\rtimes R_{\m,\Gamma}^*)$ by \cite[Corollary~1.4]{Nesh:2013}.
Explicitly, the $\sigma$-KMS$_\beta$ state $\phi_{\beta,\k,\tau}$ corresponding to a tracial state $\tau$ of $C^*(\a_{\k,1}\rtimes R_{\m,\Gamma}^*)$ is given as follows. (This explicit description comes from the proofs of \cite[Theorem~1.3~\&~Corollary~1.4]{Nesh:2013}.) For each point $[x,\a]$ in the orbit $O_\k$ of $[0,\a_{\k,1}]$, there exists $\gamma\in G_{\m,\Gamma}$ such that $\gamma[x,\a]=[0,\a_{\k,1}]$. Conjugating by $\gamma$ then defines an isomorphism of the isotropy group $(x,1)\a\rtimes R_{\m,\Gamma}^*(-x,1)$ of $[x,\a]$ onto $\a_{\k,1}\rtimes R_{\m,\Gamma}^*$, which in turn gives rise to an isomorphism of group C*-algebras $C^*((x,1)\a\rtimes R_{\m,\Gamma}^*(-x,1))\cong C^*(\a_{\k,1}\rtimes R_{\m,\Gamma}^*)$.
Let $\tau_{x,\a}$ be the tracial state of $C^*((x,1)\a\rtimes R_{\m,\Gamma}^*(-x,1))$ given by the composition of $\tau$ with the isomorphism $C^*((x,1)\a\rtimes R_{\m,\Gamma}^*(-x,1))\cong C^*(\a_{\k,1}\rtimes R_{\m,\Gamma}^*)$. Then 
\begin{equation}
\label{eqn:KMSexplicit}
\phi_{\beta,\k,\tau}(f)=\int_{O_\k}\sum_{g\in R_{\m,\Gamma}^*}f(g,[x,\a])\tau_{x,\a}(u_g) \;d\mu_{\beta,\k}([x,\a]) \quad\text{for } f\in C_c(G_{\m,\Gamma}\ltimes\Omega_R^\m).
\end{equation}
(The formula given in \eqref{eqn:KMSexplicit} can be further simplified, but shall leave it in the more compact form above for convenience.) 
A calculation shows that the map $\tau\mapsto \phi_{\beta,\k,\tau}$ from the simplex of tracial states on $C^*(\a_{\k,1}\rtimes R_{\m,\Gamma}^*)$ to the simplex of $\sigma$-KMS$_\beta$ states on $C^*(G_{\m,\Gamma}\ltimes\Omega_R^\m)$ is affine.
If $\tau$ is a convex combination $\tau=\sum_\k\lambda_\k\tau_\k$ where $\tau_\k$ is a tracial state on $C^*(\a_{\k,1}\rtimes R_{\m,\Gamma}^*)$ for each $\k$, then we let $\phi_{\beta,\k,\tau}:=\sum_\k\lambda_\k\phi_{\beta,\k,\tau_\k}$. Then the map from the simplex of tracial states of $\bigoplus_{\k\in \I_\m/i(K_{\m,\Gamma})}C^*(\a_{\k,1}\rtimes R_{\m,\Gamma}^*)$ onto the simplex of $\sigma$-KMS$_\beta$ states on $C^*(G_{\m,\Gamma}\ltimes\Omega_R^\m)$ given $\tau\mapsto\phi_{\beta,\k,\tau}$ is affine; it follows from \cite[Theorem~1.3]{Nesh:2013} that $\tau\mapsto \phi_{\beta,\k,\tau}$ is also a bijection. Weak* continuity follows from the explicit formula \eqref{eqn:KMSexplicit}; since both simplices are weak* compact, this is enough to guarantee that $\tau\mapsto\phi_{\beta,\k,\tau}$ is a homeomorphism, which concludes the proof of Theorem~\ref{thm:phasetransitions}(iii).
\end{proof}

\subsection{Ground states: the proof of part (iv)}\label{subsec:ground}

We will first use \cite[Theorem~1.9]{LLN:2019} to identify the ground states of $(C^*(G_{\m,\Gamma}\ltimes\Omega_R^\m),\mathbb{R},\sigma)$ with the states on the C*-algebra of the \emph{boundary groupoid} of the cocycle $c^N$, see \cite[Section~1]{LLN:2019} for the general definition. In our special situation, this boundary groupoid has a particularly explicit description, which is given in the following result.

\begin{proposition}\label{prop:boundarygpoid}
Let $G_{\m,\Gamma,1}:=\{(n,k)\in G_{\m,\Gamma} : N(k)=1\}$ be the kernel of the homomorphism $G_{\m,\Gamma}\to \mathbb{R}_+^*$ given by $(n,k)\mapsto N(k)$, and let 
\[
(\Omega_R^\m)_0:=\Omega_R^\m \setminus\left(\bigcup_{(n,k) :\: N(k)>1}(n,k) \Omega_R^\m\right).
\]
Then the map $\psi\mapsto \phi_\psi$ defined by 
\[
\phi_\psi(f)=\psi(f\vert_{G_{\m,\Gamma,1}\ltimes (\Omega_R^\m)_0})\quad \text{for } f\in C_c(G_{\m,\Gamma}\ltimes \Omega_R^\m)
\]
is an affine isomorphism of the state space of $C^*(G_{\m,\Gamma,1}\ltimes (\Omega_R^\m)_0)$ onto the $\sigma$-ground state space of $C^*(G_{\m,\Gamma}\ltimes \Omega_R^\m)$ where $G_{\m,\Gamma,1}\ltimes (\Omega_R^\m)_0$ is the reduction groupoid of $G_{\m,\Gamma,1}\ltimes \Omega_K^\m$ with respect to the compact subset $ (\Omega_R^\m)_0\subseteq \Omega_K^\m$.
\end{proposition}
\begin{proof}
This is a direct application of \cite[Theorem~1.9]{LLN:2019}.
\end{proof}

We are now ready for the proof of Theorem~\ref{thm:phasetransitions}(iv).

\begin{proof}[Proof of Theorem~\ref{thm:phasetransitions}(iv)]
For each class $\k\in\I_\m/i(K_{\m,\Gamma})$, let $\a_{\k,1},...,\a_{\k,k_\k}$ denote the norm-minimizing ideals in $\k$. In light of Proposition~\ref{prop:boundarygpoid}, it suffices to prove that there is an isomorphism
\[
C^*(G_{\m,\Gamma,1}\ltimes (\Omega_R^\m)_0)\cong \bigoplus_{\k\in \I_\m/i(K_{\m,\Gamma})} M_{k_\k\cdot N(\a_{\k,1})}(C^*(\a_{\k,1}\rtimes R_{\m,\Gamma}^*)).
\]
 We first claim that
\[
(\Omega_R^\m)_0=\{[\mathbf{b},\bar{\mathbf{a}}]\in \Omega_R^\m : \bar{\mathbf{a}}=\a_{\k,j}\text{ for some } \k\in \I_\m/i(K_{\m,\Gamma}), 1\leq j\leq k_\k\}.
\]
	
``$\subseteq$'': For each prime $\p\in\P_K^\m$, let $f_\p$ denote the order of the class $[\p]$ of $\p$ in $\I_\m/i(K_{\m,\Gamma})$, so that there exists $t_\p\in R_{\m,\Gamma}$ such that $\p^{f_\p}=t_\p R$. 
Let $[\mathbf{b},\bar{\mathbf{a}}]\in\Omega_R^\m$, and suppose that $v_\p(\bar{\mathbf{a}})=\infty$ for some $\p\in\P_K^\m$, so that $t_\p\bar{\mathbf{a}}=\bar{\mathbf{a}}$. By the strong approximation theorem, there exists $x\in R_\m^{-1}R_\m$ such that $v_\p(x+\mathbf{b})\geq f_\p$. Then $t_\p^{-1}(x+\mathbf{b})\in\hat{R}_S$, and 
\[
[\mathbf{b},\bar{\mathbf{a}}]=(-x,t_\p)[t_\p^{-1}(x+\mathbf{b}),\bar{\mathbf{a}}]\in (-x,t_\p)\Omega_R^\m.
\]
Since $N(t_\p)=f_\p>1$, we see that $[\mathbf{b},\bar{\mathbf{a}}]\notin(\Omega_R^\m)_0$. Therefore, if $[\mathbf{b},\bar{\mathbf{a}}]\in(\Omega_R^\m)_0$, then $v_\p(\bar{\mathbf{a}})<\infty$ for every $\p\in\P_K^\m$.
 
Next, we will show that if $[\mathbf{b},\bar{\mathbf{a}}]\in(\Omega_R^\m)_0$, then $\bar{\mathbf{a}}$ is divisible by only finitely many primes. Suppose $[\mathbf{b},\bar{\mathbf{a}}]\in\Omega_R^\m$ is such that $\{\p\in\P_K^\m : v_\p(\bar{\mathbf{a}})>0\}$ is infinite. Then there are finitely many distinct primes $\p_1,\p_2,...,\p_N$ in $\{\p\in\P_K^\m : v_\p(\bar{\mathbf{a}})>0\}$ such that the ideal $\prod_{j=1}^N\p_j$ is principal. Let $a\in R_\m$ be such that $aR=\prod_{j=1}^N\p_j$. Then $a^{-1}\bar{\mathbf{a}}$ lies in $\hat{R}_S$, and by the strong approximation theorem, there exists $x\in R_\m^{-1}R$ such that $v_\p(a^{-1}(x+\mathbf{b}))\geq 0$ for every $\p\in\P_K^\m$, so that $a^{-1}(x+\mathbf{b})\in\hat{R}_S$.  Now 
\[
[\mathbf{b},\bar{\mathbf{a}}]=(-x,a)[a^{-1}(x+\mathbf{b}),a^{-1}\bar{\mathbf{a}}]\in (-x,a)\Omega_R^\m.
\]
Since $N(a)>1$, we see that $[\mathbf{b},\bar{\mathbf{a}}]\notin(\Omega_R^\m)_0$. Thus, if $[\mathbf{b},\bar{\mathbf{a}}]\in(\Omega_R^\m)_0$, then $v_\p(\bar{\mathbf{a}})=0$ for all but finitely many $\p\in\P_K^\m$.

The above two facts imply that if $[\mathbf{b},\bar{\mathbf{a}}]\in(\Omega_R^\m)_0$, then $\bar{\mathbf{a}}=\a$ for some $\a\in\I_\m^+$. In this case, there exist $k\in K_{\m,\Gamma}$ and $1\leq j\leq k_{[\a]}$ such that $\a=k\a_{\k,j}$, and there exists $x\in R$ such that $[\mathbf{b},\a]=[x,\a]$. It remains to show that $k$ has norm $1$. Since $(\Omega_R^\m)_0$ is $G_{\m,\Gamma,1}$-invariant, we see that $(0,k)[\mathbf{0},\a_{\k,j}]=[\mathbf{0},k\a_{\k,j}]=(-x,1)[x,k\a_{\k,j}]$ lies in $(\Omega_R^\m)_0$, so we must have $N(k)=1$, that is, $\a$ must be norm-minimizing in $\k$. This finishes our proof of the first inclusion.

``$\supseteq$'':  Let $\k\in\I_\m/i(K_{\m,\Gamma})$, $1\leq j\leq k_\k$, and suppose that $[x,\a_{\k,j}]\in  \Omega_R^\m\cap (n,k)\Omega_R^\m$ for some $(n,k)\in G_{\m,\Gamma}$. There exists an integral ideal $\b\in\k$ such that $\a_{\k,j}=k\b$, so minimality of $N(\a_{\k,j})$ forces $N(k)=1$. This shows the reverse inclusion and concludes the proof of our claim.

In particular, the above claim shows that $(\Omega_R^\m)_0$ is a finite set. Moreover, there is a $G_{\m,\Gamma}$-equivariant decomposition
\[
	(\Omega_R^\m)_0=\bigsqcup_{\k\in \I_\m/i(K_{\m,\Gamma})}X_\k
\]
where $X_\k=\{[x,\a_{\k,j}] : x\in  R/\a_{\k,j}, j=1,...,k_\k\}$ is the orbit of any $[b,\a_{\k,j}]$ under the partial action of $G_{\m,\Gamma,1}$; it follows that we have the direct sum decomposition
\[
	C^*(G_{\m,\Gamma,1}\ltimes  (\Omega_R^\m)_0)\cong \bigoplus_{\k\in\I_\m/i(K_{\m,\Gamma})} C^*(G_{\m,\Gamma,1}\ltimes X_\k).
\]
For each class $\k$, $G_{\m,\Gamma,1}\ltimes X_\k$ is a transitive groupoid, and the isotropy group of the point $[0,\a_{\k,1}]$ is $\a_{\k,1}\rtimes R_{\m,\Gamma}^*$; therefore, it follows that $C^*(G_{\m,\Gamma,1}\ltimes X_\k)\cong M_{|X_\k|}(C^*(\a_{\k,1}\rtimes R_{\m,\Gamma}^*))$, see, for example, \cite[Theorem~3.1]{MRW:1987}. Since $|X_\k|=k_\k\cdot N(\a_{\k,1})$, we are done.
\end{proof}

\section{Type III$_1$ factors and the distribution of prime ideals in $\I_\m/i(K_{\m,\Gamma})$}\label{sec:typeIII1}

Each extremal $\sigma$-KMS$_\beta$ state $\phi$ on $C_\lambda^*(R\rtimes R_{\m,\Gamma})$ is a factor state, that is, the von Neumann algebra $\pi_\phi(C_\lambda^*(R\rtimes R_{\m,\Gamma}))''$ generated by the GNS representation $\pi_\phi$ of $\phi$ is a factor, see \cite[Theorem~5.3.30(3)]{BR:1997}.
It is therefore a natural problem to determine the type of the factors arising from extremal $\sigma$-KMS$_\beta$ states on $C_\lambda^*(R\rtimes R_{\m,\Gamma})$. The main result of this section is the following theorem, which, in light of the uniqueness of the injective factor of type III$_1$ with separable predual, see \cite{Con:1976} and \cite{Haa:1987}, completes the proof of Theorem~\ref{thm:phasetransitions}(ii).

\begin{theorem}\label{thm:typeIII1}
For each $\beta\in [1,2]$, let $\pi_{\phi_\beta}$ be the GNS representation of the $\sigma$-KMS$_\beta$ state $\phi_\beta$ on $C_\lambda^*(R\rtimes R_{\m,\Gamma})$ from Theorem~\ref{thm:phasetransitions}. Then the von Neumann algebra $\pi_{\phi_\beta}(C_\lambda^*(R\rtimes R_{\m,\Gamma}))''$ is an injective factor of type III$_1$ with separable predual. 
\end{theorem}

\begin{remark}\label{rmk:typeIII1}
It follows from \cite[Theorem~3.2]{LN:2011} that, for each $\beta\in[1,2]$, the $\sigma$-KMS$_\beta$ state on $C_\lambda^*(\mathbb{Z}\rtimes\mathbb{N}^\times)$ is of type III$_1$ . Moreover, it is asserted in \cite[Section~3]{Nesh:2013} that arguments analogous to those used to prove \cite[Theorem~3.2]{LN:2011} combined with \cite[Corollary~3.2]{Nesh:2011} can be used to show that, for each $\beta\in[1,2]$, the $\sigma$-KMS$_\beta$ state on $C_\lambda^*(R\rtimes R^\times)$ is of type III$_1$. 

In our more general situation, there are additional difficulties which we will overcome by using techniques from \cite[Sections~2\&3]{LagNesh:2014}. 
\end{remark}

The remainder of this section is devoted to the proof of Theorem~\ref{thm:typeIII1}.

We now briefly recall some well-known results about the flow of weights on von Neumann algebra crossed products from \cite{CT:1977} (see also \cite[Chapter~XIII~\S~2]{Tak:2003}). The general setup here is similar to that in \cite[Section~3]{LN:2011} and \cite[Section~2]{Nesh:2011}, and we will follow the notation therein.

Let $X$ be a second countable, locally compact Hausdorff space and $\mu$ a $\sigma$-finite measure on $X$. Suppose that a countably infinite discrete group $G$ acts by nonsingular transformations on the measure space $(X,\mu)$, that is, $G$ acts on $X$ by Borel automorphisms, and for each $g\in G$, the measures $\mu$ and $g\mu$ are equivalent where $g\mu$ is the push-forward of $\mu$ by $g$ defined by $g\mu(Z):=\mu(g^{-1}Z)$ for every Borel set $Z\subseteq X$. 

Assume that the action of $G$ on $(X,\mu)$ is essentially free and ergodic, so that the von Neumann algebra crossed product $L^\infty(X,\mu)\rtimes G$ is a factor. In this situation, the flow of weights has a particularly explicit description. Indeed, let $\lambda_\infty$ denote the Lebesgue measure on $\mathbb{R}_+^*$; then there are commuting actions of $G$ and $\mathbb{R}$ on $(\mathbb{R}_+^*\times X,\lambda_\infty\times\mu)$ given by 
\[
g(t,x)=\left(\frac{d g\mu}{d\mu}(gx)t,gx\right) \text{ and } s(t,x)=(e^{-s}t,x)\quad \text{for } g\in G, s\in\mathbb{R},\: (t,x)\in \mathbb{R}_+^*\times X,
\]
and the flow of weights on $L^\infty(X,\mu)\rtimes G$ is the induced action of $\mathbb{R}$ on the fixed point algebra $L^\infty(\mathbb{R}_+^*\times X,\lambda_\infty\times\mu)^G$ for the action of $G$ on $L^\infty(\mathbb{R}_+^*\times X,\lambda_\infty\times\mu)$ arising from the above action of $G$ on $(\mathbb{R}_+^*\times X,\lambda_\infty\times\mu)$ (see \cite[Chapter~XIII~\S~2~Theorem~2.23]{Tak:2003}). 
The factor $L^\infty(X,\mu)\rtimes G$ is of type III$_1$ if and only if the action of $G$ on $(\mathbb{R}_+^*\times X,\lambda_\infty\times\mu)$ is ergodic. 

We now turn to the particular case of interest to us. As before, it will be easiest to work with the C*-algebra $C^*(G_{\m,\Gamma}\ltimes\Omega_R^\m)$. Since $\phi_\beta$ factors through the expectation $\E$ onto $C(\Omega_R^\m)$ and is determined by the probability measure $\mu_\beta$, we have the following standard lemma.

\begin{lemma}\label{lem:reducetogroupmeasure}
For each $\beta\in[1,2]$, let $\tilde{\mu}_\beta$ be the unique quasi-invariant measure on $\Omega_K^\m$ that extends $\mu_\beta$ and satisfies the obvious analogue of \eqref{eqn:qi} for the action of $G_{\m,\Gamma}$ on $\Omega_K^\m$. Then
\[
\pi_{\phi_\beta}(C^*(G_{\m,\Gamma}\ltimes\Omega_R^\m))''\cong 1_{\Omega_R^\m}(L^\infty(\Omega_K^\m,\tilde{\mu}_\beta)\rtimes G_{\m,\Gamma})1_{\Omega_R^\m}.
\]
Therefore, if $L^\infty(\Omega_K^\m,\tilde{\mu}_\beta)\rtimes G_{\m,\Gamma}$ is a factor of type III$_1$, then $\pi_{\phi_\beta}(C^*(G_{\m,\Gamma}\ltimes\Omega_R^\m))''$ is also a factor of type III$_1$.
\end{lemma}

Hence, to prove Theorem~\ref{thm:typeIII1}, it suffices to show that $L^\infty(\Omega_K^\m,\tilde{\mu}_\beta)\rtimes G_{\m,\Gamma}$ is an injective factor of type III$_1$ with separable predual. 
Since $G_{\m,\Gamma}$ is amenable, $L^\infty(\Omega_K^\m,\tilde{\mu}_\beta)\rtimes G_{\m,\Gamma}$ is injective, and the separability claim is easy to see. This means that we need to prove that $L^\infty(\Omega_K^\m,\tilde{\mu}_\beta)\rtimes G_{\m,\Gamma}$ is a factor of type III$_1$.

\begin{proposition}\label{prop:factor}
For each $\beta\in [1,2]$, the action $G_{\m,\Gamma}\curvearrowright(\Omega_K^\m,\tilde{\mu}_\beta)$ is essentially free and ergodic. Hence, $L^\infty(\Omega_K^\m,\tilde{\mu}_\beta)\rtimes G_{\m,\Gamma}$ is a factor.
\end{proposition}
\begin{proof}
Arguments similar to those used in the proof of \cite[Lemma~3.3]{Nesh:2013} show that the action $G_{\m,\Gamma}\curvearrowright(\Omega_K^\m,\tilde{\mu}_\beta)$ is essentially free; note we have already made this observation in the proof of the uniqueness statement in Theorem~\ref{thm:phasetransitions}(ii).

One can argue directly using Proposition~\ref{prop:ergodicity} to show that the action $G_{\m,\Gamma}\curvearrowright(\Omega_K^\m,\tilde{\mu}_\beta)$ is ergodic. 
Alternatively, since Theorem~\ref{thm:phasetransitions}(ii) says that the state $\phi_\beta$ is the unique $\sigma$-KMS$_\beta$ state on $C^*(G_{\m,\Gamma}\ltimes\Omega_R^\m)$, \cite[Theorem~5.3.30(3)]{BR:1997} implies that $\pi_{\phi_\beta}(C^*(G_{\m,\Gamma}\ltimes\Omega_R^\m))''$ is a factor. 
Since $1_{\Omega_R^\m}$ is a full projection in $C_0(\Omega_K^\m)\rtimes G_{\m,\Gamma}$, it follows that $L^\infty(\Omega_K^\m,\tilde{\mu}_\beta)\rtimes G_{\m,\Gamma}$ is also a factor. 
Thus, the action $G_{\m,\Gamma}\curvearrowright(\Omega_K^\m,\tilde{\mu}_\beta)$ is ergodic.  
\end{proof} 

The following lemma on primes in ideal classes from $\I_\m/i(K_{\m,\Gamma})$ is the key number-theoretic result needed to compute the flow of weights on $L^\infty(\Omega_K^\m,\tilde{\mu}_\beta)\rtimes G_{\m,\Gamma}$. It is a generalization of \cite[Lemma~3.3]{Nesh:2011}. 

\begin{lemma}\label{lem:asymptotics}
Fix $\beta\in(0,1]$ and fix a class $\k\in\I_\m/i(K_{\m,\Gamma})$. For each $\lambda>1$ and each $\epsilon>0$, there exist sequences $(\p_n)_{n\geq 1}$ and $(\q_n)_{n\geq 1}$, each consisting of distinct prime ideals in $\P_K^\m$, such that 
\begin{equation}\label{eqn:asymptotics}
\left|\frac{N(\q_n)^\beta}{N(\p_n)^\beta}-\lambda\right|<\epsilon, \; \q_n\p_n^{-1}\in\k\text{ for } n\geq 1, \text{ and } \sum_{n=1}^\infty N(\p_n)^{-\beta}=\infty.
\end{equation}
\end{lemma}
\begin{proof}
The proof is similar to that of \cite[Lemma~3.3]{Nesh:2011}; it follows ideas from \cite{BZ:2000} and \cite{Bla:1977} (also see the proof of \cite[Theorem~1.2]{LagNesh:2014} for number fields).

The case where $\beta\in (0,1)$ follows from the case $\beta=1$, so it suffices to consider only the case $\beta=1$. Choose $\delta>0$ such that $1+\delta<\lambda$ and $\delta\lambda<\epsilon$. Define sets $B_n$ by
\begin{align*}
B_{2k}&:=\{\p\in\P_K^\m : \lambda^{2k}<N(\p)\leq(1+\delta)\lambda^{2k}, \p\in [R]\},\\
B_{2k+1}&:=\{\p\in\P_K^\m : \lambda^{2k+1}<N(\p)\leq(1+\delta)\lambda^{2k+1}, \p\in \k\}.
\end{align*}
By our choice of $\delta$, these sets are pairwise disjoint. For a class $\tilde{\k}\in \I_\m/i(K_{\m,\Gamma})$ and $x>0$, let 
\[
\pi_{\tilde{\k}}(x):=|\{\p\in\tilde{\k} : \p \text{ prime and } N(\p)\leq x\}|
\] 
be the number of prime ideals in the class $\tilde{\k}$ whose norms do not exceed $x$. Given functions $f$ and $g$, we shall write $f(x)\sim g(x)\text{ as } x\to\infty$ if $g(x)$ is non-zero for all sufficiently large $x$ and $\lim_{x\to\infty}\frac{f(x)}{g(x)}=1$. Note that this is equivalent to $f(x)-g(x)=o(g(x))$.
Now \cite[Chapter~VIII,~Theorem~7.2]{Mil:2013} combined with \cite[Chapter~7,~Proposition~7.17]{Nar:2004} imply that
\[
\pi_{\tilde{\k}}(x)\sim \frac{1}{h}\frac{x}{\log x} \text{ as } x\to\infty
\]
where $h:=|\I_\m/i(K_{\m,\Gamma})|$. Since we have 
\[
\frac{(1+\delta)x}{\log((1+\delta)x)}-\frac{x}{\log x}\sim \frac{\delta x}{\log x} \text{ as } x\to\infty,
\]
 it follows that
\[
\pi_{\tilde{\k}}((1+\delta)x)-\pi_{\tilde{\k}}(x)\sim \frac{\delta}{h}\frac{x}{\log x} \text{ as } x\to\infty.
\]
As this holds for every class $\tilde{\k}$, we have 
\begin{equation}\label{eqn:Bnasymptotics}
|B_n|\sim \frac{\delta}{h}\frac{\lambda^n}{n\log \lambda}\text{ as } n\to\infty.
\end{equation}
Thus, there exists $k_0$ such that $|B_{2k+1}|\geq |B_{2k}|$ for all $k\geq k_0$. Now, for each $k\geq k_0$, we can choose a subset $C_{2k+1}\subseteq B_{2k+1}$ such that $|C_{2k+1}|=|B_{2k}|$. Let $\p_1,\p_2,...$ and $\q_1,\q_2,...$ be enumerations of the sets $\bigcup_{k\geq k_0}B_{2k}$ and $\bigcup_{k\geq k_0}C_{2k+1}$, respectively, such that $N(\p_1)\leq N(\p_2)\leq\cdots$, and $N(\q_1)\leq N(\q_2)\leq\cdots$. Then if $\p_n\in B_{2k}$ for some $k\geq k_0$, we must have $\q_n\in B_{2k+1}$, in which case by our choice of $\delta$, we have
\[
N(\q_n\p_n^{-1})\in (\lambda-\epsilon,\lambda+\epsilon) \text{ and } \q_n\p_n^{-1}\in\k[R]=\k.
\]
Moreover, using \eqref{eqn:Bnasymptotics}, we see that
\[
\sum_{n=1}^\infty N(\p_n)^{-1}\geq \sum_{k=k_0}^\infty \frac{|B_{2k}|}{(1+\delta)\lambda^{2k}}=\infty.
\]
Therefore, the sequences of primes $(\p_n)_{n\geq 1}$ and $(\q_n)_{n\geq 1}$ satisfy the desired properties.
\end{proof}

Our next step is an ergodicity result that will also be used in Section~\ref{sec:multiplicativemonoid}. 
We shall need a general lemma, which we state here in the level of generality from our discussion of von Neumann algebra crossed products. Its proof is routine, so we omit it.
\begin{lemma}
\label{lem:general}
Let $X$ be a second countable, locally compact Hausdorff space and $\mu$ a $\sigma$-finite Borel measure on $X$. Suppose that a countable discrete group $G$ acts on $(X,\mu)$ by nonsingular transformations. Let $H$ be a finite index subgroup of $G$, and assume that $\tilde{\mu}$ is a measure on $G/H\times X$ such that the diagonal action $G\curvearrowright (G/H\times X,\tilde{\mu})$ is nonsingular ergodic and the restriction of $\tilde{\mu}$ to $\{H\}\times X$ coincides with $\mu$. Then the action $H\curvearrowright (X,\mu)$, obtained from the action of $G$ on $(X,\mu)$, is ergodic.
\end{lemma}

\begin{proposition}\label{prop:kmgergodic}
For each $\beta\in (0,1]$, let $\tilde{\nu}_\beta$ be the unique quasi-invariant measure on $\mathbb{A}_S/\hat{R}_S^*$ that extends $\nu_\beta$ and satisfies the obvious analogue of \eqref{eqn:qimultiplicative} for the action of $K_{\m,\Gamma}$ on $\mathbb{A}_S/\hat{R}_S^*$. Then the action of $K_{\m,\Gamma}$ on $(\mathbb{R}_+^*\times\mathbb{A}_S/\hat{R}_S^*,\lambda_\infty\times \tilde{\nu}_\beta)$ given by
\begin{equation}\label{eqn:action2}
k(t,\bar{\mathbf{a}})=(N(k)^\beta t,k\bar{\mathbf{a}})\quad \text{for } k\in K_{\m,\Gamma}, (t,\bar{\mathbf{a}})\in\mathbb{R}_+^*\times\mathbb{A}_S/\hat{R}_S^*
\end{equation}
is ergodic.  
\end{proposition}

\begin{remark}
If $\m_\infty$ is supported on all of the real embeddings of $K$ and $\m_0$ is trivial, so that $K_{\m,\Gamma}=K_+^*$ is the multiplicative subgroup of $K^*$ consisting of all (non-zero) totally positive elements, then Proposition~\ref{prop:kmgergodic} is precisely \cite[Corollary~3.2]{Nesh:2011}, which follows from Neshveyev's type computation for the high temperature KMS states on the Bost--Connes system associated with $K$, see \cite[Theorem~3.1]{Nesh:2011}.
\end{remark}

\begin{proof}[Proof of Proposition~\ref{prop:kmgergodic}]
Since the subgroup $R_{\m,\Gamma}^*$ acts trivially, the action of $K_{\m,\Gamma}$ defines an action of the quotient group $K_{\m,\Gamma}/R_{\m,\Gamma}^*$, and it suffices to show that the action of this quotient group is ergodic. 
Let $\lambda_{\m,\Gamma}$ denote the normalized Haar measure on $\I_\m/i(K_{\m,\Gamma})$. We can view $K_{\m,\Gamma}/R_{\m,\Gamma}^*$ as a subgroup of $\I_\m$; by Lemma~\ref{lem:general}, it is enough to prove that the action of $\I_\m$ on $(\I_\m/i(K_{\m,\Gamma})\times\mathbb{R}_+^*\times\mathbb{A}_S/\hat{R}_S^*,\lambda_{\m,\Gamma}\times\lambda_\infty\times \tilde{\nu}_\beta)$ given by
\[
\a([\b],t,\bar{\mathbf{a}})=([\a\b],N(\a)^\beta t,\a\bar{\mathbf{a}})\quad \text{for } \a\in \I_\m, ([\b],t,\bar{\mathbf{a}})\in \I_\m/i(K_{\m,\Gamma})\times\mathbb{R}_+^*\times\mathbb{A}_S/\hat{R}_S^*
\]
is ergodic.

Since the isomorphism $\mathbb{R}_+^*\to \mathbb{R}_+^*$ given $t\mapsto t^\beta$ preserves the measure class of $\lambda_\infty$, it suffices to show that the action of $\I_\m$ on $(\I_\m/i(K_{\m,\Gamma})\times\mathbb{R}_+^*\times\mathbb{A}_S/\hat{R}_S^*,\lambda_{\m,\Gamma}\times\lambda_\infty\times \tilde{\nu}_\beta)$ given by
\begin{equation}\label{eqn:action3}
\a([\b],t,\bar{\mathbf{a}})=([\a\b],N(\a) t,\a\bar{\mathbf{a}})\quad \text{for } \a\in \I_\m, ([\b],t,\bar{\mathbf{a}})\in \I_\m/i(K_{\m,\Gamma})\times\mathbb{R}_+^*\times\mathbb{A}_S/\hat{R}_S^*
\end{equation}
is ergodic.

Let $\R$ denote the orbit equivalence relation for the canonical action $\I_\m\curvearrowright (\mathbb{A}_S/\hat{R}_S^*,\tilde{\nu}_\beta)$ given by $\a:\bar{\mathbf{a}}\mapsto \a\bar{\mathbf{a}}$. This action is essentially free; indeed, the set
\[
\{\bar{\mathbf{a}}\in \mathbb{A}_S/\hat{R}_S^*: \text{ there exists }\p \text{ with } v_\p(\bar{\mathbf{a}})=\infty\}
\]
has $\tilde{\nu}_\beta$-measure zero by the scaling condition, and every point lying in the complement of this set has trivial isotropy. Thus, outside a set of measure zero we can define an ($\I_\m/i(K_{\m,\Gamma})\times\mathbb{R}_+^*$)-valued 1-cocycle $c$ on $\R$ by
\[
c(\bar{\mathbf{a}},\bar{\mathbf{b}})=([\a],N(\a)) \text{ if } \a\bar{\mathbf{a}}=\bar{\mathbf{b}}.
\]

Then the equivalence relation $\R(c)$ on $\I_\m/i(K_{\m,\Gamma})\times\mathbb{R}_+^*\times\mathbb{A}_S/\hat{R}_S^*$ associated with $c$ as in \cite[Section~8]{FM:1977} (see also \cite[Section~2]{LagNesh:2014}) coincides with the orbit equivalence relation for the action of $\I_\m$ on $\I_\m/i(K_{\m,\Gamma})\times\mathbb{R}_+^*\times\mathbb{A}_S/\hat{R}_S^*$ given by \eqref{eqn:action3}. Therefore, it suffices to show that $\R(c)$ is ergodic. It follows from Proposition~\ref{prop:ergodicity} and Lemma~\ref{lem:general} that $\R$ is ergodic. Hence, the results of \cite[Section~8]{FM:1977} imply that $\R(c)$ is ergodic if and only if the asymptotic range $r^*(c)$ of $c$ (\cite[Definition~8.2]{FM:1977}) coincides with $\I_\m/i(K_{\m,\Gamma})\times\mathbb{R}_+^*$, see \cite[Proposition~2.1(iii)]{LagNesh:2014}.

The proof that $r^*(c)=\I_\m/i(K_{\m,\Gamma})\times\mathbb{R}_+^*$ relies on \cite[Proposition~2.2]{LagNesh:2014} and follows the same lines as the computation of the analogous asymptotic range in the proof of \cite[Theorem~1.2]{LagNesh:2014} for number fields, but with Lemma~\ref{lem:asymptotics} used in place of \cite[Corollary~3.3]{LagNesh:2014}. 
We shall give a quick sketch of the argument here. The subset $\hat{R}_S/\hat{R}_S^*\subseteq \AA_S/\hat{R}_S^*$ is of $\tilde{\nu}_\beta$-measure one, and after removing a set of $\nu_\beta$-measure
zero, we can identify $\hat{R}_S/\hat{R}_S^*\cong\prod_{\p\in\P_K^\m}\p^{\mathbb{N}\cup\{\infty\}}$ with $\prod_{\p\in\P_K^\m}\p^{\mathbb{N}}$. The equivalence relation on $\prod_{\p\in\P_K^\m}\p^{\mathbb{N}}$ obtained by restricting $\mathcal{R}$ to $\hat{R}_S/\hat{R}_S^*$ is given by
\[
\bar{\mathbf{a}}\sim\bar{\mathbf{b}}\quad\text{if and only if } v_\p(\bar{\mathbf{a}})=v_\p(\bar{\mathbf{b}}) \text{ for all but finitely many }\p.
\]
Let $c_{\hat{R}_S/\hat{R}_S^*}$ be the restriction of $c$ to this equivalence relation; it suffices to show that $r^*(c_{\hat{R}_S/\hat{R}_S^*})$ coincides with $\I_\m/i(K_{\m,\Gamma})\times\mathbb{R}_+^*$.
The cocycle $c_{\hat{R}_S/\hat{R}_S^*}$ is of product type, as defined in \cite[Section~2]{LagNesh:2014}. Using \cite[Proposition~2.2(iii)]{LagNesh:2014}, it is enough to show that the asymptotic ratio set (see, for instance, \cite[Section~2]{LagNesh:2014}) of $c_{\hat{R}_S/\hat{R}_S^*}$ is equal to $\I_\m/i(K_{\m,\Gamma})\times\mathbb{R}_+^*$, which is done using Lemma~\ref{lem:asymptotics}.
\end{proof}

We are now ready for the proof of Theorem~\ref{thm:typeIII1}. 

\begin{proof}[Proof of Theorem~\ref{thm:typeIII1}]
Fix $\beta\in [1,2]$. In light of Lemma~\ref{lem:reducetogroupmeasure} and Proposition~\ref{prop:factor}, we only need to show that the factor $L^\infty(\Omega_K^\m,\tilde{\mu}_\beta)\rtimes G_{\m,\Gamma}$ is of type III$_1$. That is, we must show that the flow of weights on $L^\infty(\Omega_K^\m,\tilde{\mu}_\beta)\rtimes G_{\m,\Gamma}$ is trivial. Since 
\[
\frac{d(n,k)\tilde{\mu}_\beta}{d\tilde{\mu}_\beta}((n,k)w)=N(k)^\beta, \quad\text{for } (n,k)\in G_{\m,\Gamma}, w\in \Omega_K^\m,
\] 
this is equivalent to showing that the action of $G_{\m,\Gamma}$ on $(\mathbb{R}_+^*\times\Omega_K^\m,\lambda_\infty\times\tilde{\mu}_\beta)$ given by 
\begin{equation}
(n,k)(t,w)=(N(k)^\beta t,(n,k)w),\quad\text{for }(n,k)\in G_{\m,\Gamma}, w\in \Omega_K^\m
\end{equation}
is ergodic. We will now show that it suffices to prove that the action of $K_{\m,\Gamma}$ on $(\mathbb{R}_+^*\times\mathbb{A}_S/\hat{R}_S^*,\lambda_\infty\times \tilde{\nu}_{\beta-1})$ given by 
\begin{equation}
\label{eqn:ergodicreduce}
k(t,\bar{\mathbf{a}})=(N(k)^{\beta-1} t,k\bar{\mathbf{a}})\quad \text{for } k\in K_{\m,\Gamma}, (t,\bar{\mathbf{a}})\in\mathbb{R}_+^*\times\mathbb{A}_S/\hat{R}_S^*
\end{equation}
is ergodic. Our proof of this fact is a direct generalization of the special case considered in \cite[Theorem~3.2]{LN:2011}, but we include it for the convenience of the reader. Since $(\Omega_K^\m,\tilde{\mu}_\beta)$ is a quotient of $(\AA_S\times\AA_S/\hat{R}_S^*,\tilde{m}\times\tilde{\nu}_{\beta-1})$ where $\tilde{m}$ is the Haar measure on $\AA_S$ normalized such that $\tilde{m}(\hat{R}_S)=1$, it suffices to show that the action of $G_{\m,\Gamma}$ on $(\RR_+^*\times\AA_S\times\AA_S/\hat{R}_S^*,\lambda_\infty\times\tilde{m}\times\tilde{\nu}_{\beta-1})$ given by 
\begin{equation*}
(n,k)(t,\mathbf{b},\bar{\mathbf{a}})=(N(k)^\beta t,n+k\mathbf{b},k\bar{\mathbf{a}}),\quad\text{for }(n,k)\in G_{\m,\Gamma}, t\in\RR, \mathbf{b}\in\AA_S, \bar{\mathbf{a}}\in\AA_S/\hat{R}_S^*
\end{equation*}
is ergodic. Since $R_\m^{-1}R$ is dense in $\AA_S$ by the strong approximation theorem, the action of $R_\m^{-1}R$ on $(\AA_S,\tilde{m})$ by translation is ergodic. Thus, any $(R_\m^{-1}R\times\{1\})$-invariant measurable function on the product space $\RR_+^*\times\AA_S\times\AA_S/\hat{R}_S^*$ does not depend on the second coordinate. Hence, to prove that the above action is ergodic, it suffices to prove that the action given in Equation~\eqref{eqn:ergodicreduce} is ergodic.

For $\beta=1$, this follows since $\tilde{\nu}_0=\delta_{\bar{\mathbf{0}}}$ and $\{N(k) : k\in K_{\m,\Gamma}\}=\mathbb{Q}_+^*$, which is dense in $\mathbb{R}_+^*$, whereas for $\beta\in(1,2]$, this follows from Proposition~\ref{prop:kmgergodic}.
\end{proof}

\begin{remark}
Since the seminal work of Bost and Connes \cite{BC:1995}, there have been several operator algebraic constructions from number theory that lead to C*-dynamical systems exhibiting interesting phase transitions where the high temperature KMS states are factor states of type III$_1$. See, for example, \cite{BC:1995}, \cite{CMR:2005}, \cite{HL:1997}, \cite{BZ:2000}, \cite{Nesh:2002,Nesh:2011}, \cite{LN:2011}, and \cite{LNT:2013}.
We remark that in all cases, uniqueness of the high temperature KMS states boils down to that fact that certain $L$-functions do not have poles at $1$, and the crucial number-theoretic result needed to compute the type is a version of the prime number theorem. 
\end{remark}

\section{The boundary quotient}\label{sec:boundaryquotient}

By \cite[Theorem~7.1]{Bru:2019}, the C*-algebra $C_\lambda^*(R\rtimes R_{\m,\Gamma})$ has a unique maximal ideal $I_{\P_K^\m}$. The boundary quotient of $C_\lambda^*(R\rtimes R_{\m,\Gamma})$, as defined in \cite[Section~7]{Li2:2013} (see also \cite[Chapter~5.7]{CELY:2017}), is the quotient $C_\lambda^*(R\rtimes R_{\m,\Gamma})/I_{\P_K^\m}$.

Moreover, \cite[Theorem~7.1]{Bru:2019} gives an explicit description of the ideal $I_{\P_K^\m}$. We shall only need to know that $I_{\P_K^\m}$ corresponds to the subset of $\Omega_R^\m$ given by 
\[
\Omega_R^\m\setminus(\hat{R}_S\times\{\bar{\mathbf{0}}\})= \{[\mathbf{b},\bar{\mathbf{a}}] \in\Omega_R^\m: \text{ there exists } \p\in\P_K^\m \text{ with } v_\p(\bar{\mathbf{a}})<\infty\}.
\] 

Let $\rho: C_\lambda^*(R\rtimes R_{\m,\Gamma})\to C_\lambda^*(R\rtimes R_{\m,\Gamma})/I_{\P_K^\m}$ be the quotient map. For each $t\in\mathbb{R}$, the automorphism $\sigma_t$ leaves $I_{\P_K^\m}$ invariant, so $\sigma$ defines a time evolution $\bar{\sigma}$ on $C_\lambda^*(R\rtimes R_{\m,\Gamma})/I_{\P_K^\m}$ such that $\bar{\sigma}_t(\rho(\lambda_{(b,a)}))=N(a)^{it}\rho(\lambda_{(b,a)})$ for all $(b,a)\in R\rtimes R_{\m,\Gamma}$.

\begin{theorem}\label{thm:bqKMS}
The C*-dynamical system $(C_\lambda^*(R\rtimes R_{\m,\Gamma})/I_{\P_K^\m},\mathbb{R},\bar{\sigma})$ has a unique $\bar{\sigma}$-KMS$_1$ state $\bar{\phi}$, and there are no $\bar{\sigma}$-KMS$_\beta$ states for $\beta\neq 1$. 
Moreover, if $\pi_{\bar{\phi}}$ is the GNS representation of $\bar{\phi}$, then $\pi_{\bar{\phi}}(C_\lambda^*(R\rtimes R_{\m,\Gamma})/I_{\P_K^\m})''$ is isomorphic to the injective factor of type III$_1$ with separable predual, and $\bar{\phi}$ is determined by the values 
\begin{equation}\label{eqn:bqKMS}
\bar{\phi}(\rho(E_{(x+\a)\times (\a\cap R_{\m,\Gamma})}))=N(\a)^{-1}\quad\text{for all }x\in R \text{ and } \a\in\I_\m^+.
\end{equation}
\end{theorem}
\begin{proof}
If $\phi$ is a $\bar{\sigma}$-KMS$_\beta$ on $C_\lambda^*(R\rtimes R_{\m,\Gamma})/I_{\P_K^\m}$, then the composition $\phi\circ\rho$ is a $\sigma$-KMS$_\beta$ state on $C_\lambda^*(R\rtimes R_{\m,\Gamma})$ that vanishes on $\ker\rho=I_{\P_K^\m}$. Moreover, the map $\rho^*:\phi\mapsto \phi\circ\rho$ is injective.
Suppose that a $\sigma$-KMS$_\beta$ state $\phi$ on $C_\lambda^*(R\rtimes R_{\m,\Gamma})$ belongs to the range of $\rho^*$, and let $\mu$ be the quasi-invariant probability measure on $\Omega_R^\m$ determined by $\phi\vert_{C(\Omega_R^\m)}$. 
From our analysis of quasi-invariant probability measures in Section~\ref{sec:equilibrium}, we know that either $\mu=\mu_\beta$ (in the case $\beta\in[1,2]$) or $\mu$ is a convex combination of the measures $\mu_{\beta,\k}$, $\k\in\I_\m/i(K_{\m,\Gamma})$ (in the case $\beta\in (2,\infty)$), where $\mu_\beta$ is defined in the proof of Theorem~\ref{thm:phasetransitions}(ii) and $\mu_{\beta,\k}$ is defined in the proof of Theorem~\ref{thm:phasetransitions}(iii).
Since $\phi\vert_{C(\Omega_R^\m)}$ vanishes on $I_{\P_K^\m}\cap C(\Omega_R^\m)=C_0(\Omega_R^\m\setminus (\hat{R}_S\times\{\bar{\mathbf{0}}\}))$, it follows that $\mu$ is concentrated on $\hat{R}_S\times\{\bar{\mathbf{0}}\}$. 
This happens only for $\beta=1$ and $\mu=\mu_1$, in which case $\phi=\phi_1$ is the unique $\sigma$-KMS$_1$ state on $C_\lambda^*(R\rtimes R_{\m,\Gamma})$ from Theorem~\ref{thm:phasetransitions}(ii). This implies that there are no $\bar{\sigma}$-KMS$_\beta$ states for $\beta\neq 1$.

A direct calculation using \cite[Corollary~8.14.4]{Ped:2018} shows that $\phi_1$ vanishes on the ideal $I_{\P_K^\m}$ and thus factors through $\rho$ to define a $\bar{\sigma}$-KMS$_1$ state on $C_\lambda^*(R\rtimes R_{\m,\Gamma})/I_{\P_K^\m}$; the uniqueness of $\bar{\sigma}$-KMS$_1$ states now follows from the injectivity of $\rho^*$.

Since $\pi_{\bar{\phi}}\circ\rho=\pi_{\phi_1}$ where $\pi_{\phi_1}$ is the GNS representation of $\phi_1$, it follows from  Theorem~\ref{thm:phasetransitions}(ii) that $\pi_{\bar{\phi}}(C_\lambda^*(R\rtimes R_{\m,\Gamma})/I_{\P_K^\m})''$ is isomorphic to the injective factor of type III$_1$ with separable predual.
\end{proof}

\begin{remark}
For the case of trivial $\m$ and $\Gamma$, the uniqueness claim in Theorem~\ref{thm:bqKMS} follows from \cite[Theorem~6.7]{CDL:2013}.
\end{remark}

\section{Phase transitions on C*-algebras of multiplicative monoids}\label{sec:multiplicativemonoid}

For each $a\in R_{\m,\Gamma}$, let $\lambda_a$ denote the isometry on $\ell^2(R_{\m,\Gamma})$ determined by $\lambda_a(\epsilon_x)=\epsilon_{ax}$ where $\{\epsilon_x : x\in R_{\m,\Gamma}\}$ is the canonical orthonormal basis for $\ell^2(R_{\m,\Gamma})$. Then the left regular C*-algebra of the (commutative) semigroup $R_{\m,\Gamma}$ is the sub-C*-algebra of $\mathcal{B}(\ell^2(R_{\m,\Gamma}))$ generated by these isometries, that is, 
\[
C_\lambda^*(R_{\m,\Gamma}):=C^*(\{\lambda_a : a\in R_{\m,\Gamma}\}).
\]

The C*-algebra $C_\lambda^*(R_{\m,\Gamma})$ also carries a canonical time evolution $\sigma^\times$ that is determined on the generating isometries by $\sigma_t^\times(\lambda_a)=N(a)^{it}\lambda_a$ for $a\in R_{\m,\Gamma}$. 

\begin{remark}
Using \cite[Proposition~3.9]{Bru:2019} and Li's theory of semigroup C*-algebras from \cite{Li1:2012,Li2:2013}, one can show that there is an injective *-homomorphism 
\[
C_\lambda^*(R_{\m,\Gamma})\hookrightarrow C_\lambda^*(R\rtimes R_{\m,\Gamma})
\] 
such that $\lambda_a\mapsto \lambda_{(0,a)}$ for all $a\in R_{\m,\Gamma}$. Hence, under this embedding, the time evolution $\sigma^\times$ coincides with the restriction of the time evolution $\sigma$ to (the image of) $C_\lambda^*(R_{\m,\Gamma})$.
\end{remark}

The (commutative) semigroup $R_{\m,\Gamma}/R_{\m,\Gamma}^*$ can be identified with the semigroup of principal ideals that are generated by an element from $R_{\m,\Gamma}$. For each $a\in R_{\m,\Gamma}$, let $\lambda_{aR_{\m,\Gamma}^*}$ denote the corresponding isometry in the left regular C*-algebra $C_\lambda^*(R_{\m,\Gamma}/R_{\m,\Gamma}^*)$; this C*-algebra also carries a canonical time evolution, which we also denote by $\sigma^\times$. It is determined by $\sigma^\times(\lambda_{aR_{\m,\Gamma}^*})=N(a)^{it}\lambda_{aR_{\m,\Gamma}^*}$ for $aR_{\m,\Gamma}^*\in R_{\m,\Gamma}/R_{\m,\Gamma}^*$.

In this section, we briefly explain how the techniques used to prove Theorem~\ref{thm:phasetransitions} also lead to phase transition theorems for the C*-dynamical systems 
\[
(C_\lambda^*(R_{\m,\Gamma}),\mathbb{R},\sigma^\times)\quad \text{and}\quad(C_\lambda^*(R_{\m,\Gamma}/R_{\m,\Gamma}^*),\mathbb{R},\sigma^\times).
\] 
Namely, we have the following two theorems, the first one for the left regular C*-algebra of a congruence monoid itself, and the second one for left regular C*-algebra of a semigroup of principal ideals that are generated by elements from a congruence monoid.

\begin{theorem}\label{thm:ptmultiplicative}
Let $K$ be a number field, $\m$ a modulus for $K$, and $\Gamma$ a subgroup of $(R/\m)^*$.
\begin{enumerate}[\upshape(i)]
		\item There are no $\sigma^\times$-KMS$_\beta$ states on $C_\lambda^*(R_{\m,\Gamma})$ for $\beta<0$.
		\item The simplex of $\sigma^\times$-KMS$_0$ states on $C_\lambda^*(R_{\m,\Gamma})$ is isomorphic to the simplex of $\sigma^\times$-invariant states on the commutative group C*-algebra $C^*(K_{\m,\Gamma})$.
		\item For each $\beta\in (0,1]$, the simplex of $\sigma^\times$-KMS$_\beta$ states on $C_\lambda^*(R_{\m,\Gamma})$ is isomorphic to the simplex of states on the commutative group C*-algebra $C^*(R_{\m,\Gamma}^*)$. Moreover, if $\psi_{\beta,\chi}$ is the extremal $\sigma^\times$-KMS$_\beta$ state corresponding to the character $\chi\in \widehat{R_{\m,\Gamma}^*}$ and $\pi_{\psi_{\beta,\chi}}$ is the GNS representation of $\psi_{\beta,\chi}$, then $\pi_{\psi_{\beta,\chi}}(C_\lambda^*(R_{\m,\Gamma}))''$ is isomorphic to the injective factor of type III$_1$ with separable predual.
		\item For each $\beta>1$, the simplex of $\sigma^\times$-KMS$_\beta$ states on $C_\lambda^*(R_{\m,\Gamma})$ is isomorphic to the simplex of states on the commutative C*-algebra
		\[
		\bigoplus_{\k\in\I_\m/i(K_{\m,\Gamma})}C^*(R_{\m,\Gamma}^*).
		\]
		\item The set of $\sigma^\times$-ground states on $C_\lambda^*(R_{\m,\Gamma})$ is isomorphic to the state space of the C*-algebra 
		\[
			\bigoplus_{\k\in \I_\m/i(K_{\m,\Gamma})} M_{k_\k}(C^*(R_{\m,\Gamma}^*))
		\]
		where $k_\k$ is the number of norm-minimizing ideals in the class $\k$.
\end{enumerate}
\end{theorem}

\begin{theorem}\label{thm:ptmultiplicative2}
Let $K$ be a number field, $\m$ a modulus for $K$, and $\Gamma$ a subgroup of $(R/\m)^*$.
\begin{enumerate}[\upshape(i)]
		\item There are no $\sigma^\times$-KMS$_\beta$ states on $C_\lambda^*(R_{\m,\Gamma}/R_{\m,\Gamma}^*)$ for $\beta<0$.
		\item The simplex of $\sigma^\times$-KMS$_0$ states on $C_\lambda^*(R_{\m,\Gamma})$ is isomorphic to the simplex of $\sigma^\times$-invariant states on the commutative group C*-algebra $C^*(K_{\m,\Gamma}/R_{\m,\Gamma}^*)$.
		\item For each $\beta\in (0,1]$, there is a unique $\sigma^\times$-KMS$_\beta$ state $\omega_\beta$ on $C_\lambda^*(R_{\m,\Gamma}/R_{\m,\Gamma}^*)$. Moreover, if $\pi_{\omega_\beta}$ is the GNS representation of $\omega_\beta$, then $\pi_{\omega_\beta}(C_\lambda^*(R_{\m,\Gamma}/R_{\m,\Gamma}^*))''$ is isomorphic to the injective factor of type III$_1$ with separable predual.
		\item For each $\beta>1$, the simplex of $\sigma^\times$-KMS$_\beta$ states on $C_\lambda^*(R_{\m,\Gamma}/R_{\m,\Gamma}^*)$ is isomorphic to the simplex of states on the finite-dimensional commutative C*-algebra $\mathbb{C}^{h_{\m,\Gamma}}$ where $h_{\m,\Gamma}:=|\I_\m/i(K_{\m,\Gamma})|$.
		\item The set of $\sigma^\times$-ground states on $C_\lambda^*(R_{\m,\Gamma})$ is isomorphic to the state space of the C*-algebra 
		\[
			\bigoplus_{\k\in \I_\m/i(K_{\m,\Gamma})} M_{k_\k}(\mathbb{C})
		\]
		where $k_\k$ is the number of norm-minimizing ideals in the class $\k$.
\end{enumerate}
\end{theorem}

\begin{remark}\label{rmk:ptmultiplicative}
\begin{itemize}
\item[(a)] For the special case of trivial $\m$ and $\Gamma$, the parameterization results in Theorem~\ref{thm:ptmultiplicative}(i)-(iv) were already asserted in \cite[Remark~7.5]{CDL:2013}. 
\item[(b)] An alternative approach to computing the $\sigma^\times$-KMS$_\beta$ states on $C_\lambda^*(R^\times)$ and $C_\lambda^*(R^\times/R^*)$ for $\beta>1$ is given in \cite[Remark~6.6.5]{CELY:2017}. Presumably, the approach taken there could also be used  to compute the low temperature KMS states on $C_\lambda^*(R_{\m,\Gamma})$ and $C_\lambda^*(R_{\m,\Gamma}/R_{\m,\Gamma}^*)$.
\item[(c)] Using the canonical isomorphisms $C^*(K_{\m,\Gamma})\cong C(\widehat{K_{\m,\Gamma}})$ and $C^*(R_{\m,\Gamma}^*)\cong C(\widehat{R_{\m,\Gamma}^*})$ given by the Fourier transform, the parameterizations in Theorem~\ref{thm:ptmultiplicative}(iii)\&(iv) can be phrased in terms of characters of the discrete abelian groups $K_{\m,\Gamma}$ and $R_{\m,\Gamma}^*$.
Specifically, 
\begin{itemize}
\item for each $\beta\in (0,1]$, the extremal $\sigma^\times$-KMS$_\beta$ states on $C_\lambda^*(R_{\m,\Gamma})$ are parameterized by the characters of the discrete abelian group $R_{\m,\Gamma}^*$;
\item for each $\beta>1$, the extremal $\sigma^\times$-KMS$_\beta$ states on $C_\lambda^*(R_{\m,\Gamma})$ are parameterized by pairs $(\k,\chi)$ where $\k$ is a class in $\I_\m/i(K_{\m,\Gamma})$ and $\chi$ is character of $R_{\m,\Gamma}^*$.
\end{itemize}
\item[(d)] An analogues statement involving characters of the discrete abelian group $K_{\m,\Gamma}/R_{\m,\Gamma}^*$ holds for the parameterization given by Theorem~\ref{thm:ptmultiplicative2}(ii).
\end{itemize}
\end{remark}

The arguments needed to prove these theorems are almost identical, so we will only give a proof of Theorem~\ref{thm:ptmultiplicative}. 

\begin{proof}[Proof of Theorem~\ref{thm:ptmultiplicative}]
The strategy is similar to that used to prove Theorem~\ref{thm:phasetransitions}, so we will only give a sketch of the arguments.
There is a canonical action of the group $K_{\m,\Gamma}$ on $\mathbb{A}_S/\hat{R}_S^*$, and the C*-algebra of the reduction groupoid 
\[
K_{\m,\Gamma}\ltimes \hat{R}_S/\hat{R}_S^*=\{(k,\bar{\mathbf{a}})\in K_{\m,\Gamma}\times\hat{R}_S/\hat{R}_S^* :k\bar{\mathbf{a}}\in\hat{R}_S/\hat{R}_S^*\}\subseteq K_{\m,\Gamma}\ltimes \mathbb{A}_S/\hat{R}_S^*
\]
carries a canonical time evolution, which we also denote by $\sigma^\times$, determined by the real-valued 1-cocycle $c^\times: K_{\m,\Gamma}\ltimes \hat{R}_S/\hat{R}_S^*\to \mathbb{R}_+^*$ given by $(k,\bar{\mathbf{a}})\mapsto N(k)$. 
Arguments analogous to those given in \cite[Section~5]{Bru:2019} show that the C*-algebra $C_\lambda^*(R_{\m,\Gamma})$ can be canonically and $\mathbb{R}$-equivariantly identified with the groupoid C*-algebra $C^*(K_{\m,\Gamma}\ltimes \hat{R}_S/\hat{R}_S^*)$.
Hence, it suffices to compute all KMS and ground states of the C*-dynamical system $(C^*(K_{\m,\Gamma}\ltimes \hat{R}_S/\hat{R}_S^*),\mathbb{R},\sigma^\times)$.

A short calculation similar to that from the proof of Theorem~\ref{thm:phasetransitions}(i) shows that assertion (i) holds. 

For $\beta\in [0,1]$, Theorem~\ref{thm:uniquemeasure} asserts that the measure $\nu_\beta$ defined in Section~\ref{subsec:part(ii)} is the unique probability measure on $\hat{R}_S/\hat{R}_S^*$ that satisfies
\begin{equation*}\label{eqn:qimult}
\nu(kZ)=N(k)^{-\beta}\nu(Z) 
\end{equation*}
for all $k\in K_{\m,\Gamma}$ and Borel sets $Z\subseteq\hat{R}_S/\hat{R}_S^*$ such that $kZ\subseteq\hat{R}_S/\hat{R}_S^*$. For $\beta=0$, we have $\nu_\beta=\delta_{\bar{\mathbf{0}}}$, and the isotropy group of the point $\bar{\mathbf{0}}$ is all of $K_{\m,\Gamma}$. Since a state $\tau$ of $C^*(K_{\m,\Gamma})$ is $\sigma^\times$-invariant if and only if $\tau(u_k)=0$ for all $k\in K_{\m,\Gamma}$ with $N(k)\neq 1$, assertion (ii) follows from \cite[Theorem~1.3]{Nesh:2013} and \cite[Corollary~1.4]{Nesh:2013}. 

Now suppose $\beta\in(0,1]$. Then the measure $\nu_\beta$ is concentrated in the set 
\[
A:=\{\bar{\mathbf{a}}\in\hat{R}_S/\hat{R}_S^* : \bar{\mathbf{a}}_\p\neq 0\text{ for all }\p\}.
\]
Since the isotropy group of any point in this set is equal to $R_{\m,\Gamma}^*$, the parameterization result asserted in (iii) follows from \cite[Theorem~1.3]{Nesh:2013} (an argument similar to that used in the proof of Theorem~\ref{thm:phasetransitions}(iii) is needed to verify that the given parameterization is an isomorphism of simplexes).

The state $\psi_{\beta,\chi}$ corresponding to the character $\chi\in\widehat{R_{\m,\Gamma}^*}$ is given explicitly by
\begin{equation}
\label{eqn:multKMSformula}
\psi_{\beta,\chi}(f)=\int_A\sum_{g\in R_{\m,\Gamma}^*}\chi(g)f(g,\bar{\mathbf{a}})\;d\nu_\beta(\bar{\mathbf{a}}) \quad\text{ for } f\in C_c(K_{\m,\Gamma}\ltimes \hat{R}_S/\hat{R}_S^*).
\end{equation}

(Note that this explicit formula is not given in the statement of \cite[Theorem~1.3]{Nesh:2013}, but is given in its proof, which uses \cite[Theorem~1.1]{Nesh:2013}.)
Inspired by \cite[Proposition~5.2]{LLNSW:2015}, which came from an idea of Neshveyev \cite[Remark~2.5]{Nesh:2013}, we shall now describe the von Neumann algebra $\pi_{\psi_{\beta,\chi}}(C^*(K_{\m,\Gamma}\ltimes\hat{R}_S/\hat{R}_S^*))''$ generated by the GNS representation $\pi_{\psi_{\beta,\chi}}$ of $\psi_{\beta,\chi}$. 
Let $\chi\in\widehat{R_{\m,\Gamma}^*}$, and choose an extension $\tilde{\chi}$ of $\chi$ to $K_{\m,\Gamma}$. There is a *-homomorphism 
\[
\tilde{\Psi}_{\tilde{\chi}}:C_0(\mathbb{A}_S/\hat{R}_S^*)\rtimes K_{\m,\Gamma}\to L^\infty(\AA_S/\hat{R}_S^*,\tilde{\nu}_\beta)\rtimes (K_{\m,\Gamma}/R_{\m,\Gamma}^*)
\]
such that $\tilde{\Psi}_{\tilde{\chi}}(fu_k)=\tilde{\chi}(k)fu_{\bar{k}}$ for all $f\in C_0(\AA_S/\hat{R}_S^*)$ and $k\in K_{\m,\Gamma}$, where $\bar{k}$ denotes the image of $k$ under the quotient map $K_{\m,\Gamma}\to K_{\m,\Gamma}/R_{\m,\Gamma}^*$ and $\tilde{\nu}_\beta$ is the measure on $\mathbb{A}_S/\hat{R}_S^*$ from the statement of Proposition~\ref{prop:kmgergodic}. 

Let $\Psi_{\tilde{\chi}}$ denote the composition
\[
C^*(K_{\m,\Gamma}\ltimes\hat{R}_S/\hat{R}_S^*)\cong 1_{\hat{R}_S/\hat{R}_S^*}(C_0(\mathbb{A}_S/\hat{R}_S^*)\rtimes K_{\m,\Gamma})1_{\hat{R}_S/\hat{R}_S^*}\to L^\infty(\AA_S/\hat{R}_S^*,\tilde{\nu}_\beta)\rtimes (K_{\m,\Gamma}/R_{\m,\Gamma}^*)
\]
where the second arrow is the restriction of $\tilde{\Psi}_{\tilde{\chi}}$ to the (full) corner $1_{\hat{R}_S/\hat{R}_S^*}(C_0(\mathbb{A}_S/\hat{R}_S^*)\rtimes K_{\m,\Gamma})1_{\hat{R}_S/\hat{R}_S^*}$. 
A calculation using the explicit formula for $\psi_{\beta,\chi}$ given in Equation~\eqref{eqn:multKMSformula} shows that $\psi_{\beta,\chi}=\varphi\circ\Psi_{\tilde{\chi}}$ where $\varphi$ is the canonical normal state on $1_{\hat{R}_S/\hat{R}_S^*}(L^\infty(\AA_S/\hat{R}_S^*,\tilde{\nu}_\beta)\rtimes (K_{\m,\Gamma}/R_{\m,\Gamma}^*)1_{\hat{R}_S/\hat{R}_S^*}$ determined by $\nu_\beta$. Since the image of $\Psi_{\tilde{\chi}}$ is strong operator dense in the corner $1_{\hat{R}_S/\hat{R}_S^*}(L^\infty(\AA_S/\hat{R}_S^*,\tilde{\nu}_\beta)\rtimes (K_{\m,\Gamma}/R_{\m,\Gamma}^*)1_{\hat{R}_S/\hat{R}_S^*}$, we get a (non-canonical) isomorphism
\[
\pi_{\psi_{\beta,\chi}}(C^*(K_{\m,\Gamma}\ltimes\hat{R}_S/\hat{R}_S^*))''\cong 1_{\hat{R}_S/\hat{R}_S^*}(L^\infty(\mathbb{A}_S/\hat{R}_S^*,\tilde{\nu}_\beta)\rtimes K_{\m,\Gamma}/R_{\m,\Gamma}^*)1_{\hat{R}_S/\hat{R}_S^*}
\]
The assertion about injectivity and separability is easy to see, and factoriality follows from extremality of $\psi_{\beta,\chi}$ by \cite[Theorem~5.3.30(3)]{BR:1997}.
To prove our assertion about type, it suffices to show that the flow of weights on $L^\infty(\mathbb{A}_S/\hat{R}_S^*,\tilde{\nu}_\beta)\rtimes (K_{\m,\Gamma}/R_{\m,\Gamma}^*)$ is trivial, and for this it is enough to show that the action of $K_{\m,\Gamma}/R_{\m,\Gamma}^*$ on 
\[
(\mathbb{R}_+^*\times\mathbb{A}_S/\hat{R}_S^*,\lambda_\infty\times \tilde{\nu}_\beta)
\] 
given by
\[
\bar{k}(t,\bar{\mathbf{a}})=(N(k)^\beta t,k\bar{\mathbf{a}})\quad \text{for } \bar{k}\in K_{\m,\Gamma}/R_{\m,\Gamma}^*, (t,\bar{\mathbf{a}})\in\mathbb{R}_+^*\times\mathbb{A}_S/\hat{R}_S^*
\]
is ergodic. This follows from Proposition~\ref{prop:kmgergodic}.

For $\beta\in (1,\infty)$, Lemma~\ref{lem:qimeasuresforlargebeta} says that the extremal probability measures that satisfy \eqref{eqn:qimultiplicative} are precisely the measures $\{\nu_{\beta,\k} : \k\in \I_\m/i(K_{\m,\Gamma})\}$. These measures are concentrated in the set 
\[
\{\bar{\mathbf{a}} : \bar{\mathbf{a}}_\p\neq 0 \text{ for all }\p \}.
\]
Since the isotropy group of any point in this set is $R_{\m,\Gamma}^*$, the parameterization stated in Theorem~\ref{thm:ptmultiplicative}(iv) also follows from \cite[Theorem~1.3]{Nesh:2013}, and arguing as in Theorem~\ref{thm:phasetransitions}(iii), one shows that this parameterization is an isomorphism of simplexes.

Following the proof of Theorem~\ref{thm:phasetransitions}(iv), we see that the boundary set of the cocycle $c^\times$ (cf. \cite[Section~1]{LLN:2019}) is equal to 
\[
(\hat{R}_S/\hat{R}_S^* )_0:=\{\bar{\mathbf{a}}\in\hat{R}_S/\hat{R}_S^* : \bar{\mathbf{a}}=\a_{\k,j} \text{ for some } 1\leq j\leq k_\k\}
\]
where $\a_{\k,1},...,\a_{\k,k_\k}$ are the norm-minimizing ideals in the class $\k$ (see the discussion preceding Theorem~\ref{thm:phasetransitions}). Let $K_{\m,\Gamma,1}:=\{x\in K_{\m,\Gamma} : N(x)=1\}$. 
Then \cite[Theorem~1.9]{LLN:2009}  asserts that the map $\psi\mapsto \phi_\psi$ defined by 
\[
\phi_\psi(f)=\psi(f\vert_{K_{\m,\Gamma,1}\ltimes (\hat{R}_S/\hat{R}_S^* )_0})\quad \text{for } f\in C_c(K_{\m,\Gamma}\ltimes \hat{R}_S/\hat{R}_S^* )
\]
is an affine isomorphism of the state space of $C^*(K_{\m,\Gamma,1}\ltimes (\hat{R}_S/\hat{R}_S^* )_0)$ onto the $\sigma$-ground state space of $C^*(K_{\m,\Gamma}\ltimes \hat{R}_S/\hat{R}_S^* )$ where $K_{\m,\Gamma,1}\ltimes (\hat{R}_S/\hat{R}_S^*)_0$ is the reduction groupoid of $K_{\m,\Gamma,1}\ltimes \mathbb{A}_S/\hat{R}_S^*$ with respect to the subset $(\hat{R}_S/\hat{R}_S^* )_0\subseteq \mathbb{A}_S/\hat{R}_S^*$. 
Now arguments similar to those used to prove Theorem~\ref{thm:phasetransitions}(iv) show that
\[
C^*(K_{\m,\Gamma,1}\ltimes (\hat{R}_S/\hat{R}_S^* )_0)\cong \bigoplus_{\k\in \I_\m/i(K_{\m,\Gamma})} M_{k_\k}(C^*(R_{\m,\Gamma}^*)),
\]
which finishes the proof of Theorem~\ref{thm:ptmultiplicative}(v). 
\end{proof}

\end{document}